\documentclass[reqno]{amsart}
\usepackage{amssymb,latexsym,amscd} 
\usepackage{amsmath}
\makeatletter
\renewcommand{\pod}[1]{\allowbreak\mathchoice
 {\if@display \mkern 18mu\else \mkern 8mu\fi (#1)}
 {\if@display \mkern 18mu\else \mkern 8mu\fi (#1)}
 {\mkern4mu(#1)}
 {\mkern4mu(#1)}
}
\usepackage{color}

\usepackage{enumerate}[1.)]
\usepackage{pb-diagram} 
\renewcommand{\eqref}[1]{(\ref{#1})} 
\usepackage{enumitem}
\theoremstyle{plain}
\newtheorem{theorem}{Theorem}[section]

\newtheorem{corollary}[theorem]{Corollary}

\newtheorem{lemma}[theorem]{Lemma}

\newtheorem{remark}{Remark}
\newtheorem{proposition}[theorem]{Proposition}
\newtheorem{algorithm}{Algorithm}

\newtheorem{definition}[theorem]{Definition}

\theoremstyle{remark}

\newcommand {\Q}{{\mathbb{Q}}}

\newcommand {\Z}{{\mathbb{Z}}}
\newcommand {\N}{{\mathbb{N}}}

\newcommand {\LL}{{\mathcal{L}}}

\newcommand{\rk} {{\mathop{\rm rk}}}

\numberwithin{equation}{section}

\begin{document}
\title{On the $4$--rank of class groups of Dirichlet biquadratic fields}
\date{\today}
\author{\'Etienne Fouvry}
\address{Universit\' e Paris--Saclay, CNRS, Laboratoire de math\'ematiques d'Orsay, 91405 Orsay, France}
\email{Etienne.Fouvry@universite-paris-saclay.fr}

\author{Peter Koymans}
\address{Max Planck Institute for Mathematics, Vivatsgasse 7, 53111 Bonn, Germany}
\email{koymans@mpim-bonn.mpg.de}

\author{Carlo Pagano}
\address{Max Planck Institute for Mathematics, Vivatsgasse 7, 53111 Bonn, Germany}
\email{carlein90@gmail.com}


\begin{abstract}
We show that for $100\%$ of the odd, squarefree integers $n > 0$ the $4$--rank of $\text{Cl}(\Q(i, \sqrt{n}))$ is equal to $\omega_3(n) - 1$, where $\omega_3$ is the number of prime divisors of $n$ that are $3$ modulo $4$.
\end{abstract}

\maketitle

\section{Introduction}
A classical result in number theory is the theorem of Gauss on ambiguous binary quadratic forms. This theorem gives, in modern terms, a description of $\text{Cl}(K)[2]$ if $K$ is a quadratic extension of $\Q$. In particular, Gauss proved that the dimension of the $\mathbb{F}_2$-vector space $\text{Cl}(K)[2]$ equals $\omega(\Delta_K) - 1$, where $\omega(\cdot)$ is the number of distinct  prime divisors, $\text{Cl}(K)$ is the narrow class group and $\Delta_K$ is the discriminant of $K$. Since then, the class group has taken a prominent role in number theory but still remains a rather mysterious object.

From a heuristic standpoint, the class group is better understood in families of number fields due to the conjectures of Cohen--Lenstra \cite{Coh-Len} and later Cohen--Martinet \cite{Coh-Mart}. The Cohen--Martinet heuristics have several known flaws, and they have been corrected and extended by several authors \cite{BL, LMZ, WW}. To state the Cohen--Lenstra conjectures, let $p$ be an odd prime and let $A$ be a finite, abelian $p$-group. Then the $p$-part of the class group of an imaginary quadratic field is conjectured to be isomorphic to $A$ with probability proportional to $1/\# \text{Aut}(A)$.

For $p = 2$ this obviously breaks down due to the rather predictable nature of $\text{Cl}(K)[2]$. A natural workaround was found by Gerth, who predicted that a finite abelian $2$-group $A$ is isomorphic to $2\text{Cl}(K)[2^\infty]$ with probability proportional to $1/\# \text{Aut}(A)$. This was recently proven by Smith \cite{Smith}. The odd part remains a mystery with the most significant result due to Davenport--Heilbronn \cite{Dav-Hei} on the first moment, which was later improved by Bhargava--Shankar--Tsimerman \cite{Bha-Sha-Tsi}. There are also results on class groups of cubic fields \cite{Bha-Var} and $S_n$-fields \cite{Ho-Sha-Var}.

In this paper we are precisely interested in the case where the degree of the number field is not coprime to the part of the class group we are studying. This case is excluded in the heuristics of Cohen--Lenstra and
 Cohen--Martinet, and we hope that this work will aid in the development of heuristics in this case. Other known 
results regarding statistical properties of class groups in families where the degree is not coprime to the part of the class group are due to Fouvry--Kl\"uners \cite{Fou-Klu-Inv, Fou-Klu-Ann, Fou-Klu-PLMS, Fou-Klu-IMRN}, Klys \cite{Klys}, Koymans--Pagano \cite{Koy-Pag2} and Pagano--Sofos \cite{Pag-Sof}, who developed heuristics for ray class groups based 
on earlier work of Varma \cite{Var} and proved them for the $4$-rank of imaginary quadratic fields.

To state the main result we use the following notations. For $\ell \in \{1, 3\}$ and $n \geq 1$ an integer, let $\omega_\ell (n)$ be the number of distinct prime factors of $n$ which are congruent to $\ell$ modulo $4$. We define $K_n := \Q(i, \sqrt{n})$ and we let $\text{Cl}(K_n)$ be the class group of $K_n$. These fields were first studied by Dirichlet, in the context of quadratic forms \cite{Diri}, and further studied for special values of $n$ by Azizi et al., see for instance \cite{Azizi1, Azizi2, Azizi3,  Azizi4, Azizi5, Azizi6, Azizi7}, and also in \cite{Fou-Koy}. Furthermore, the $2^k$--rank of a finite abelian group $A$ is by definition $\rk_{2^k} A := \dim_{\mathbb{F}_2} 2^{k - 1} A/ 2^k A$.

\begin{theorem}
\label{main}
We have
\[
\#\{0 < n < x : n \textup{ odd and squarefree}, \ \rk_4 \textup{Cl}(K_n) \neq \omega_3(n) - 1\} = O\left(\frac{x}{(\log x)^{1/8}}\right).
\]
\end{theorem}

In simple terms, we have that the $4$--rank of $\text{Cl}(K_n)$ is $\omega_3(n) - 1$ for $100\%$ of the odd, squarefree integers $n$. Note that this behavior is wildly different from the case of quadratic extensions of $\Q$ (see \cite[Prop. 12]{Fou-Koy} for instance), and we believe it to be a non-trivial task to develop appropriate heuristics in this setting. A weaker result can be found in \cite[Theorem 1]{Fou-Koy}, where it is proven that at least $28\%$ of the odd, squarefree integers $n$ satisfy $\rk_4 \textup{Cl}(K_n) = \omega_3(n) - 1$. 

To prove Theorem \ref{main}, we start by giving a description of $\text{Cl}(K_n)^\vee[2]$. Such a description can be obtained from the work of Fr\"ohlich \cite{Frohlich}, who studied $\text{Cl}(K)^\vee[2]$ for any biquadratic extension $K$ of $\Q$. This was later extended to $\text{Cl}(K)^\vee[2]$ with $K$ an arbitrary multiquadratic extension of $\Q$ in \cite{Koy-Pag}.

Once we described $\text{Cl}(K_n)^\vee[2]$, we are in the position to obtain a criterion for an element of $\text{Cl}(K_n)^\vee[2]$ to be in $2\text{Cl}(K_n)^\vee[4]$. To do so, we introduce the notion of genericity.

\begin{definition}
We say that an odd, squarefree integer $n > 0$ is generic if it has a prime divisor $5$ modulo $8$.
\end{definition}

This notion of generic integer already appears in \cite{Fou-Koy}. More precisely, for $n$ odd and squarefree, we have the equality
(see \cite[Prop. 8]{Fou-Koy})
$$
\rk_2 (\text{Cl} (K_n)) =
\begin{cases}
2 \omega_1 (n) +\omega_3 (n) -1 \text{ if } n \text{ is generic},\\
2\omega_1 (n) +\omega_3 (n)-2 \text{ if } n \text{ is not generic.}
\end{cases}
$$

As we shall see below, our algebraic criterion is only valid if $n$ is generic. It is here that we make essential use that $\Q(i)$ has class number $1$, and it is plausible that Theorem \ref{main} can be extended to any family of the shape $\Q(\sqrt{d}, \sqrt{n})$ as long as $d$ is negative and $\Q(\sqrt{d})$ has class number $1$. It would be most interesting to extend the results further in case that $\Q(\sqrt{d})$ does not have class number $1$.

We shall reserve the letter $\pi$ for irreducible elements in $\Z [i]$. For $n\geq 3$ an odd squarefree integer, we introduce the following arithmetical function $f(n)$ defined by 
\begin{multline}
\label{deff1}
f(n) := \frac 14 \ \cdot \sharp \Bigl\{\beta \in \mathbb Z [i] :\ \beta \equiv \pm 1 \bmod 4\Z [i], \ \beta \vert n\text{ such that } \\ \text{ for all } \pi \vert \beta \text{ the Gaussian integer } n/\beta \text { is a square modulo } \pi \\
\text{ and for all } \pi \vert (n/\beta) \text{ the Gaussian integer } \beta \text { is a square modulo } \pi
\Bigr\}.
\end{multline}
This function resembles the quantity appearing in \cite[Lemma 16]{Fou-Klu-Inv}. The definition of $f(n)$ directly implies that it is 
a power of $2$  satisfying the inequalities
$$
1\leq f(n) \leq 2^{2\omega_1 (n)+\omega_3 (n)-1}.
$$
We can now state our key algebraic result.

\begin{theorem}
\label{tAlgebra}
Let $n$ be generic. Then we have
\[
2^{\rk_4 \textup{Cl}(K_n)} = f(n)
\]
and furthermore $f(n) \geq 2^{\omega_3(n) - 1}$.
\end{theorem}

With the aid of Magma we computed a list of generic integers $3 \leq n \leq 1000$ for which $\rk_4 \textup{Cl}(K_n) \geq \omega_3(n)$. By Theorem \ref{tAlgebra} we certainly have for such $n$ that $\rk_4 \textup{Cl}(K_n) \geq \omega_3(n) - 1$. This gives the following table (we have excluded those with $\omega_3(n) = 0$, since they trivially satisfy $\rk_4 \textup{Cl}(K_n) \geq \omega_3(n)$) \\

\begin{center}
\begin{tabular}{| l | l | l |}
\hline
 $\omega_3(n)$ & Generic $3 \leq n \leq 1000$ with $\rk_4 \textup{Cl}(K_n) \geq \omega_3(n)$\\
 \hline
1 & \{39,
55,
95,
111,
155,
183,
203,
259,
295,
299,
327,
355,
371,
395,\\
&
407,
471,
543,
559,
583,
655,
663,
667,
687,
695,
755,
763,
831,
895,\\
&
915,
955,
995\} \\
\hline
2 & \{777, 897\}. \\
\hline
\end{tabular}
\end{center}

\vspace{0.4cm}
The above table shows $33$ integers, while the total number of generic integers satisfying $3 \leq n \leq 1000$ and $\omega_3(n) > 0$ is $96$ of which $78$ are with $\omega_3(n) = 1$ and $18$ are with $\omega_3(n) = 2$. Furthermore, the smallest generic $n$ with $\omega_3(n) > 0$ and $\rk_4 \textup{Cl}(K_n) \geq \omega_3(n) + 1$ is $n = 1443$, and the smallest $n$ with instead $\rk_4 \textup{Cl}(K_n) \geq \omega_3(n) + 2$ is $n = 4895$.











It is not hard to show from our methods that we always have the inequalities
\begin{align}
\label{eInequalities}
2^{\omega_3(n) - 2} \leq \frac{f(n)}{2} \leq 2^{\rk_4 \textup{Cl}(K_n)} \leq f(n)
\end{align}
for odd, squarefree $n \geq 3$. Since our focus is Theorem \ref{main}, we shall not include the proofs of these inequalities. Our main analytic result shows that, for a special type of averaging, $f(n)$ is close to $2^{\omega_3 (n)-1}$. We have

\begin{theorem} 
\label{central} 
Uniformly for $x \geq 2$ we have 
\begin{equation}
\label{importantsum}
\sum_{n\leq x} \mu^2 (2n) \Bigl( \frac {f(n)} { 2^{\omega_3 (n) -1}}\Bigr) = \sum_{n \leq x} \mu^2 (2n) + O (x \log^{-1/8}x).
\end{equation}
\end{theorem}

Standard methods from analytic number theory show the equality
\begin{equation}
\label{summu2}
\sum_{n \leq x} \mu^2 (2n) = \frac 4{\pi^2}\cdot x + O (\sqrt x),
\end{equation}
uniformly for $x \geq 2$. This shows that \eqref{importantsum} is an asymptotic formula. 

With slightly more effort, particularly in the proof of Lemma \ref{>7negligible}, it is possible to improve the error term in \eqref{importantsum} to $ O (x \log^{-\theta}x)$ for any $\theta <1/4$. The equality \eqref{importantsum} could be generalized to the set of integers $n\leq x$ such that all the prime divisors of $n$ belong to an imposed congruence class.

The layout of the paper is as follows. In \S \ref{s2Torsion} we study the 2--torsion of $\text{Cl}(K_n)$. Then, in \S \ref{s4Torsion}, we derive our pivotal algebraic results, culminating in the proof of Theorem \ref{tAlgebra}. The next sections are devoted to the analysis of the sum appearing in \eqref{importantsum}. In our final section \S \ref{sMain} we show how Theorem \ref{tAlgebra} and Theorem \ref{central} imply Theorem \ref{main}.

\section*{Acknowledgements}
We thank Florent Jouve for some numerical computations. We are also very grateful to Mark Shusterman for useful discussions regarding the function field version of this problem. Peter Koymans and Carlo Pagano wish to thank the Max Planck Institute for Mathematics in Bonn for its financial support, great work conditions and an inspiring atmosphere.

\section{On the $2$--torsion of $\text{C{\rm l}}(K_n)$}
\label{s2Torsion}
For an abelian group $A$, we write $A[m]$ for the part of $A$ that is killed by $m$ and $A^\vee := \text{Hom}(A, \mathbb{C})$ for its dual. In what follows, we let $n \in \mathbb{Z}_{\geq 3}$ be an odd, squarefree integer. Recall that $n$ is \emph{generic} in case there exists a prime number congruent to $5$ modulo $8$ that divides $n$. Write $n:=p_1 \cdot \ldots \cdot p_r \cdot q_1 \cdot \ldots \cdot q_s$, where for each $(h,k) \in [r] \times [s]$ we have that $p_h$ and $q_k$ are respectively $1$ and $3$ modulo $4$. For each $h \in [r]$, decompose $p_h$ in $\mathbb{Z}[i]$ as
$$
p_h := \pi_h \cdot \overline{\pi_h},
$$
where $\pi_h:=a_h+ib_h$ with $2 \mid b_h$ and $a_h \equiv 1 \bmod 4$. The following gives a complete description of the quadratic extensions of $K_n$ that are unramified at all finite places and thus, since $K_n$ is totally complex, a description of the space $\text{Cl}(K_n)^{\vee}[2]$ by class field theory.

\begin{proposition} 
\label{genus theory}
Let $L/K_n$ be a quadratic extension. Then it is unramified if and only if there exist functions
$$
\epsilon_1, \epsilon_2: [r] \to \{0,1\}, \alpha: [s] \to \{0,1\}
$$
such that 
\begin{align}
\label{eSumMod2}
\sum_{h \in [r]: p_h \equiv 5 \bmod 8} \Bigl(\epsilon_1(h)+\epsilon_2(h) \Bigr)\equiv 0 \bmod 2
\end{align}
and
$$
L=\mathbb{Q}\left(i, \sqrt{n}, \sqrt{\prod_{h \in [r]} \pi_h^{\epsilon_1(h)} \overline{\pi_h}^{\epsilon_2(h)} \cdot \prod_{k \in [s]} q_k^{\alpha (k)}}\right).
$$
\end{proposition}

\begin{proof}
It is well known that $\mathbb{Z}[i]$ is a principal ideal domain. It follows that the generator of $\text{Gal}(K_n/\mathbb{Q}(i))$ acts as $-\text{id}$ on $\text{Cl}(K_n)$. In particular, every unramified abelian extension of $K_n$ remains Galois over $\mathbb{Q}(i)$. Furthermore, we also know that the extension $K_n/\mathbb{Q}(i)$ must ramify at some finite place $v$ of $\mathbb{Q}(i)$. Hence an inertia group at $v$ in $\text{Gal}(L/\mathbb{Q}(i))$ must be of order $2$ and project non-trivially in $\text{Gal}(K_n/\mathbb{Q}(i))$. It follows that $L/\mathbb{Q}(i)$ is a biquadratic extension, in other words there must be $\gamma \in \mathbb{Q}(i)^{*}$ with $L = K_n(\sqrt{\gamma})$. 

Next we claim that if we have a finite place $v$ of $\mathbb{Q}(i)$ with $v(\gamma)$ odd, then it must be that $v(n)>0$. Indeed, $K_n/\mathbb{Q}(i)$ is unramified at all places $v$ with $v(n)=0$; observe that this is also correct at $1 + i$, since $n$ is a rational integer. But $\mathbb{Q}(i, \sqrt{\gamma})/\mathbb{Q}(i)$ will certainly ramify at $v$ in case $v(\gamma)$ is odd and hence the extension $K_n(\sqrt{\gamma})/K_n$ will ramify at any place of $K_n$ above $v$. This shows our claim. 

Thanks to the last step and since $\mathbb{Z}[i]$ is a PID, we can suppose that $\gamma$ is an element of $\mathbb{Z}[i]$ that divides $n$ in $\mathbb{Z}[i]$.

Now let $a+ib$ be an element of $\mathbb{Z}[i]$, with $a \not\equiv b \bmod 2$, i.e. $a+ib$ is coprime to $1+i$. Then we claim that $\mathbb{Q}(i, \sqrt{a + ib})$ is unramified at $1 + i$ if and only if $4 \mid b$. To this end, we recall that elements of $\mathbb{Z}_2[i]$ of the shape $1 + (1 + i) u$ or $1 + (1 + i)^3 u$, with $u \in \mathbb{Z}_2[i]^\ast$, yield ramified quadratic extensions of $\mathbb{Q}_2(i)$. We first show that $2 \mid b$. Suppose not. Then, by our assumption on $a+ib$, it must be that $a$ is even. Then we can rewrite $a + ib = 2a' + 2ib' + i = 2(a' + ib') + i$ with $a', b'$ integers. This can be rewritten as $1 + (1 + i) u$ with $u \in \mathbb{Z}_2[i]^\ast$. 

So we must have that $2 \mid b$ and $a \equiv 1 \bmod 2$. Furthermore, since $-1$ is a square in $\mathbb{Q}(i)$, we can assume that $a$ is $3$ modulo $4$. Now suppose that $4$ does not divide $b$. Hence we can rewrite $a + ib = a + 2ib' = -1 + 4a' + 2ib'$, where $a', b'$ are integers and $b'$ is odd, which equals $-1 + 2i + 4z = 1 - 2(1 - i) + 4z$ with a $z \in \mathbb{Z}[i]$. This has the shape $1+(1+i)^3u$ with $u$ a unit in $\mathbb{Z}_2[i]$. Therefore it yields a ramified extension of $\mathbb{Q}_2(i)$ and the desired claim is proved. 

We have obtained that $\gamma=a+ib$ is a divisor of $n$ with $a$ odd and $b$ divisible by $4$. Observe furthermore that, since $-1$ is a square in $\mathbb{Q}(i)$, we can reduce to the case that $a$ is $1$ modulo $4$. Now it is straightforward to check that $\gamma$ is precisely one of the elements listed above. Reversely, it is easy to check that all such $\gamma$ give an unramified extension.
\end{proof}

We denote by $\text{Gn}(K_n)$ the span of the elements listed in Proposition \ref{genus theory} in $\frac{\mathbb{Q}(i)^{*}}{\mathbb{Q}(i)^{*2}}$. More precisely, these are the elements
\[
\prod_{h \in [r]} \pi_h^{\epsilon_1(h)} \overline{\pi_h}^{\epsilon_2(h)} \cdot \prod_{k \in [s]} q_k^{\alpha (k)}
\] 
as $\epsilon_1, \epsilon_2: [r] \to \{0,1\}, \alpha: [s] \to \{0,1\}$ varies and satisfies equation (\ref{eSumMod2}).

\section{A criterion for the $4$-torsion for generic $n$}
\label{s4Torsion}
We shall now establish a general fact that will be the key tool to exploit the condition of genericity on $n$. To do so, we start by recalling the inflation--restriction exact sequence. Let $G$ be a profinite group, $N$ a normal open subgroup and $A$ a discrete $G$-module. Note that $G/N$ naturally acts on $A^N$. Then we have an exact sequence
\[
0 \rightarrow H^1(G/N, A^N) \rightarrow H^1(G, A) \rightarrow H^1(N, A)^{G/N} \rightarrow H^2(G/N, A^N) \rightarrow H^2(G, A),
\]
where the first and fourth map are inflation, the second map is restriction and the third map is transgression. We remark that $G$ naturally acts on $H^1(N, A)$ by sending a cocycle $f: N \rightarrow A$ to $(g \cdot f)(n) = g \cdot f(g^{-1} n g)$ and this action descends to an action of $G/N$.

For a field $K$, we denote by $G_K$ the absolute Galois group of $K$. If $L/K$ is any finite Galois extension of fields of characteristic different from $2$, we apply the inflation--restriction sequence with $G = G_K$, $N = G_L$ and $A = \mathbb{F}_2$ with trivial action. There is, by Kummer theory, an isomorphism
$$
H^1(N, A)^{G/N} \cong \left(\frac{L^\ast}{{L^{\ast}}^2}\right)^{\text{Gal}(L/K)}.
$$
The map from right to left is given by sending $\alpha$ to the character $\chi_\alpha$, which is by definition the character corresponding to $\sqrt{\alpha}$. Combining this with the inflation--restriction exact sequence, we hence obtain a natural map
$$
r: \left(\frac{L^\ast}{{L^{\ast}}^2}\right)^{\text{Gal}(L/K)} \to H^2(\text{Gal}(L/K),\mathbb{F}_2),
$$
whose kernel consists precisely of the image of $K^{*}$ in $L^\ast/{L^{\ast}}^2$ and whose image consists precisely of those classes in $H^2(\text{Gal}(L/K), \mathbb{F}_2)$ that become trivial when inflated to $H^2(G_K, \mathbb{F}_2)$. We start with a lemma.

\begin{lemma}
\label{lLocal}
Let $E$ be a local field of characteristic $0$ and let $F/E$ be an unramified extension. Then the inflation map 
$$
H^2(\emph{Gal}(F/E), \mathbb{F}_2) \rightarrow H^2(G_E, \mathbb{F}_2)
$$
is the zero map.
\end{lemma}

\begin{proof}
This is a special case of \cite[Prop. 4.4] {Koy-Pag2}.
\end{proof}

We can now prove the following proposition, which is based on ideas from \cite[Proposition~4.10]{Koy-Pag}. We say that a class $\theta \in H^2(\text{Gal}(L/\mathbb{Q}(i)), \mathbb{F}_2)$ is \emph{locally trivial} at a place $v$ of $\Q(i)$ if $\theta$ is trivial in $H^2(G_{\Q(i)_v}, \mathbb{F}_2)$.

\begin{proposition} 
\label{cleaning the ramification for Q(i)}
Let $L/\mathbb{Q}(i)$ be a Galois $2$-extension of $\mathbb{Q}(i)$ and take $p$ to be a rational prime that is congruent to $5$ modulo $8$. Suppose that $L$ ramifies at both places of $\mathbb{Q}(i)$ lying above $p$. Assume that $1 + i$ is unramified in $L/\mathbb{Q}(i)$. Let $\theta \in H^2(\emph{Gal}(L/\mathbb{Q}(i)), \mathbb{F}_2)$ be such that the inflation of $\theta$ to $G_{\mathbb{Q}(i)}$ is locally trivial at all the places of $\mathbb{Q}(i)$ which ramify in $L/\mathbb{Q}(i)$. Suppose furthermore that for each odd place $v$ of $\mathbb{Q}(i)$, the class $\theta$ restricted to an inertia subgroup $I_v$ of $\emph{Gal}(L/\mathbb{Q}(i))$ yields a trivial element of $H^2(I_v, \mathbb{F}_2)$. 

Then there exists $\alpha \in \left(\frac{L^\ast}{{L^{\ast}}^2}\right)^{\emph{Gal}(L/\mathbb{Q}(i))}$ with $r(\alpha)=\theta$ and $L(\sqrt{\alpha})/L$ unramified. 
\end{proposition}

\begin{proof}
We first claim that there exists $\alpha$ such that $r(\alpha)=\theta$ and that the extension $L(\sqrt{\alpha})/L$ is unramified above any odd place. Consider the exact sequence
$$
1 \rightarrow \{\pm 1\} \rightarrow \overline{\Q(i)}^\ast \rightarrow \overline{\Q(i)}^\ast \rightarrow 1,
$$
where the last map is squaring. Taking Galois cohomology and using Hilbert 90, we deduce that there is an injection
$$
0 \rightarrow H^2(G_{\Q(i)}, \mathbb{F}_2) \rightarrow H^2(G_{\Q(i)}, \overline{\Q(i)}^\ast).
$$
Then, by class field theory, we have another injection
$$
0 \rightarrow H^2(G_{\Q(i)}, \overline{\Q(i)}^\ast) \rightarrow \bigoplus_{v \in \Omega_{\Q(i)}} H^2(G_{\Q(i)_v}, \overline{\Q(i)_v}^\ast),
$$
where $\Omega_{\Q(i)}$ are the places of $\Q(i)$. Hence, to check if $\theta$ is trivial in $H^2(G_{\Q(i)}, \mathbb{F}_2)$, we can check this locally in $H^2(G_{\Q(i)_v}, \overline{\Q(i)_v}^\ast)$ for every $v \in \Omega_{\Q(i)}$. By assumption $\theta$ is trivial locally at all places $v$ that ramify in $L/\Q(i)$. Furthermore, Lemma \ref{lLocal} shows that $\theta$ is trivial at the unramified places, and hence we have shown that $\theta$ is trivial in $H^2(G_{\Q(i)}, \mathbb{F}_2)$. We deduce that there is $\alpha$ with $r(\alpha) = \theta$.

Our next task is to adjust the ramification at the odd places. Suppose that $L(\sqrt{\alpha})/L$ is ramified at some odd place $w$ of $L$. If $v$ is the place of $\Q(i)$ below $w$, and $v$ is unramified in $L/\Q(i)$, then we twist by $\chi_\pi$ with $\pi$ a prime element of $\Z[i]$ corresponding to $v$. Since $\alpha$ is invariant modulo squares, this ensures that $v$ is unramified in $L(\sqrt{\alpha})/L$ without changing the ramification at any other odd place.

Now suppose instead that the place $v$ below $w$ is ramified in $L/\Q(i)$. We filter $L_w/\Q(i)_v$ as a tower $L_w/K/\Q(i)_v$, where $K$ is the largest unramified extension of $\Q(i)_v$ inside $L_w$. The assumption that $\theta$ is trivial when restricted to $I_v$ precisely implies, by the inflation--restriction sequence, that $\chi_\alpha$ equals the restriction of some central character $\chi$ from $G_K$. Since $v$ is an odd place, such characters are in the span of the unramified character of $K$ and a ramified character of $\Q(i)_v$. Therefore the extension $L(\sqrt{\alpha})/L$ is automatically unramified at $v$ for any choice of $\alpha$ with $r(\alpha) = \theta$.

Having established the claim, it remains to adjust the ramification at $1 + i$. Let $w$ be a place of $L$ above $1 + i$. By assumption $L_w/\mathbb{Q}_2(i)$ is unramified. Therefore the Galois group $\text{Gal}(L_w/\mathbb{Q}_2(i))$ is cyclic and thus $H^2(\text{Gal}(L_w/\mathbb{Q}_2(i)), \mathbb{F}_2)$ is cyclic of order $2$. The non-trivial element in $H^2(\text{Gal}(L_w/\mathbb{Q}_2(i)), \mathbb{F}_2)$ is realized via an unramified extension. Hence there exists $c \in \mathbb{Q}(i)^{*}$ such that $c \alpha$ yields an unramified class of $\frac{L_w^\ast}{L_w^{\ast 2}}$ for all choices of $w$ above $v$. Furthermore, since $\alpha$ is invariant modulo squares, the same $c$ will work simultaneously for all places $w$ of $L$ above $1 + i$. 

Let now $p = \pi \overline{\pi}$ be a factorization of our prime $p \equiv 5 \bmod 8$ in $\mathbb{Z}[i]$. Observe that multiplying $\alpha$ by elements in the span of $\{\pi, \overline{\pi}, i, 1 + i\}$ changes only the ramification at the places above $2$. Indeed, this follows from the assumption that $\theta$ is trivial when restricted to $H^2(I_v, \mathbb{F}_2)$. Now $p \equiv 5 \bmod 8$ implies that $\{\pi, \overline{\pi}, i, 1 + i\}$ is a basis of $\frac{\mathbb{Q}_2(i)^\ast}{{\mathbb{Q}_2(i)^\ast}^2}$. Hence $c$ can be picked in the space $\langle \{\pi, \overline{\pi}, i, 1 + i\} \rangle$, cleaning the ramification precisely at the places above $1 + i$ without affecting any other place of $L$, which concludes our proof.
\end{proof}

Let $n \in \mathbb{Z}_{\geq 3}$ be odd, squarefree and generic. We can now describe the space $2\text{Cl}(K_n)^{\vee}[4]$. For any place $v$ of $\mathbb{Q}(i)$ we denote by $(-,-)_v$ the Hilbert symbol pairing defined on $\mathbb{Q}(i)_v^\ast$ and attaining values in $\{1,-1\}$. Recall that for $x,y \in \mathbb{Q}(i)_v^\ast$, one has that $(x,y)_v=1$ if and only if $\chi_x \cup \chi_y$ yields a trivial class in $H^2(G_{\mathbb{Q}(i)_v},\mathbb{F}_2)$.

\begin{proposition} 
\label{4rk space for odd n} 
Let $n \in \mathbb{Z}_{\geq 3}$ be odd, squarefree and generic. Let $\alpha$ be in $\emph{Gn}(K_n)$. Then the character $\chi_{\alpha}$ is in $2\emph{Cl}(K_n)^{\vee}[4]$ if and only if for any finite place $v$ with $v(n) \neq 0$ one has 
$$
\left(\alpha,\frac{n}{\alpha}\right)_v=1.
$$
\end{proposition}

\begin{remark}
\label{rGeneric}
The forward implication will not use that $n$ is generic, but for the other implication this will be crucial.
\end{remark}

\begin{proof}
Observe that the elements $\alpha$ in $\text{Gn}(K_n)$ with $\chi_{\alpha} \in 2\text{Cl}(K_n)^{\vee}[4]$ are, if non-trivial, precisely those with $K_n(\sqrt{\alpha})$ contained in a cyclic degree $4$ unramified extension of $K_n$. Such an extension is Galois over $\mathbb{Q}(i)$ as we argued in Proposition \ref{genus theory}. Furthermore, picking an inertia element at a place dividing $\alpha$ and one at a place dividing $\frac{n}{\alpha}$ gives a lift by involutions of the basis of $\text{Gal}(K_n(\sqrt{\alpha})/\mathbb{Q}(i))$ dual to $\{\chi_\alpha, \chi_{\frac{n}{\alpha}}\}$. This forces $L$ to be $K_n(\sqrt{\alpha}, \sqrt{\beta})$ with 
$$
\beta \in \left(\frac{K_n(\sqrt{\alpha})^\ast}{{K_n(\sqrt{\alpha})^\ast}^2}\right)^{\text{Gal}(K_n(\sqrt{\alpha})/\mathbb{Q}(i))}
$$
and 
$$
r(\beta) = \chi_\alpha \cup \chi_{\frac{n}{\alpha}}.
$$ 
Conversely any time we realize, via the map $r$, this class via an unramified quadratic extension of $K_n(\sqrt{\alpha})$ we conclude that $\chi_\alpha \in 2\text{Cl}(K_n)^{\vee}[4]$. 

Hence, for the only if part, we see that $\chi_\alpha \cup \chi_{\frac{n}{\alpha}}$ must be in the image of $r$. It follows that $\chi_\alpha \cup \chi_{\frac{n}{\alpha}}$ is in the kernel of the inflation to $H^2(G_{\mathbb{Q}(i)}, \mathbb{F}_2)$ by the inflation--restriction exact sequence mentioned at the beginning of this section. But then $\chi_\alpha \cup \chi_{\frac{n}{\alpha}}$ must be locally trivial at all places $v$, and this implies precisely that $\left(\alpha,\frac{n}{\alpha}\right)_v=1$ for all finite places $v$ with $v(n) \neq 0$.

While for the if part, we apply Proposition \ref{cleaning the ramification for Q(i)} with $\theta := \chi_\alpha \cup \chi_{\frac{n}{\alpha}}$. Since $\left(\alpha,\frac{n}{\alpha}\right)_v=1$ for all finite places $v$ with $v(n) \neq 0$ by assumption, we see that $\theta$ is locally trivial at all places of $\Q(i)$ that ramify in $L := K_n(\sqrt{\alpha})$. Furthermore, the shape of $\theta$ shows that the class $\theta$ restricted to an inertia subgroup $I_v$ of $\text{Gal}(L/\mathbb{Q}(i))$ yields a trivial element of $H^2(I_v, \mathbb{F}_2)$ for each odd place $v$ of $\mathbb{Q}(i)$. 

The fact that $n$ is generic ensures that we have a prime $p \equiv 5 \bmod 8$ to which we can apply Proposition \ref{cleaning the ramification for Q(i)}. Then Proposition \ref{cleaning the ramification for Q(i)} gives us the required $\beta$.
\end{proof}

\begin{corollary}
\label{c4rkbig}
Suppose that $n \geq 3$ is odd, squarefree and generic. Then we have
\[
\dim_{\mathbb{F}_2} 2\textup{Cl}(K_n)^{\vee}[4] \geq \omega_3(n) - 1.
\]
\end{corollary}

\begin{proof}
First of all, we have that
\[
\dim_{\mathbb{F}_2} \textup{Cl}(K_n)^{\vee}[2] = 2\omega_1(n) + \omega_3(n) - 2
\]
by Proposition \ref{genus theory} or \cite[Proposition 8]{Fou-Koy}. We now consider the linear map $T: \textup{Cl}(K_n)^{\vee}[2] \rightarrow \{1, -1\}^{\{v \mid n\}}$ that sends $\chi_\alpha$ to $\left\{\left(\alpha,\frac{n}{\alpha}\right)_v\right\}_v$, where $v$ runs through all finite places dividing $n$. If $v$ corresponds to a prime $p \equiv 3 \bmod 4$ in $\Z$ and $\alpha \in \Z$, then we have that
$$
\left(\alpha, \frac{n}{\alpha}\right)_p = (\alpha, n)_p = 1.
$$
Combining this with Hilbert reciprocity, we see that the image of $T$ has dimension at most $2\omega_1(n) - 1$. But, for generic $n$, the kernel of $T$ is precisely $2\textup{Cl}(K_n)^{\vee}[4]$ by Proposition \ref{4rk space for odd n}. Hence the lemma follows from the rank--nullity theorem.
\end{proof}

We can now prove Theorem \ref{tAlgebra}.

\begin{proof}[Proof of Theorem \ref{tAlgebra}]
Take $n > 0$ to be odd and squarefree. There is a natural surjective map
\[
\{\beta \in \Z[i] : \beta \equiv \pm 1 \bmod 4 \Z[i], \beta \mid n\} \rightarrow \text{Gn}(K_n).
\]
The kernel is given by $-1$ and is hence of size $2$. There is also a natural map $\text{Gn}(K_n) \rightarrow \text{Cl}(K_n)^\vee[2]$, given by sending $\beta$ to $\chi_\beta$, with kernel given by $\chi_n$, again of size $2$. By Proposition \ref{4rk space for odd n} it follows that for generic $n$
\[
2^{\rk_4 \text{Cl}(K_n)} = \frac{1}{2}\left|\left\{\alpha \in \text{Gn}(K_n) : \left(\alpha, \frac{n}{\alpha}\right)_v = 1 \text{ for all } v\right\}\right|.
\]
The condition that $(\alpha, n/\alpha)_v = 1$ for all $v$ is equivalent to: for every $\pi \mid \alpha$, we have that $n/\alpha$ is a square modulo $\pi$, and for every $\pi \mid n/\alpha$, we have that $\alpha$ is a square modulo $\pi$. This shows that
\[
2^{\rk_4 \text{Cl}(K_n)} = f(n).
\]
The theorem then follows from Corollary \ref{c4rkbig}.
\end{proof}

\begin{remark}
It is now now easy to prove two of the three inequalities in equation (\ref{eInequalities}). The bound $2^{\rk_4 \textup{Cl}(K_n)} \leq f(n)$ follows Proposition \ref{4rk space for odd n} and Remark \ref{rGeneric}. Furthermore, the proof of Corollary \ref{c4rkbig} shows that
\[
f(n) \geq 2^{\omega_3(n) - 1}
\]
without any assumptions on $n$. The final inequality
\[
\frac{f(n)}{2} \leq 2^{\rk_4 \textup{Cl}(K_n)}
\]
is substantially trickier and we shall only sketch it. From the material here one sees that if $\alpha \in \textup{Gn}(K_n)$ is such that $(\alpha, n/\alpha)_v = 1$ for all $v$, then one can pick a non-trivial point on the conic
\[
x^2 = \alpha y^2 + \frac{n}{\alpha} z^2
\]
such that the extension $\Q(i, \sqrt{n}, \sqrt{\alpha}, \sqrt{x + \sqrt{\alpha} y})/\Q(i, \sqrt{n}, \sqrt{\alpha})$ is only ramified at $2$. Then some local considerations at $2$ finish the proof.
\end{remark}

\section{Convention, definitions and classical lemmas} 
We now pass to the proof of Theorem \ref{central}.

\subsection{Gaussian integers} 
We will follow several conventions that appear in \cite[Chapt. 9.7 \& 9.8]{Ir-Ro} concerning the ring $\Z [i]$ of Gaussian integers. The multiplicative group of its {\it units} is denoted by $\mathbb U:= \{ \pm 1, \pm i\}$. A Gaussian integer $\alpha$ is said to be {\it odd} if its norm ${\rm N} (\alpha) := {\rm N}_{\Q(i)/\Q}(\alpha)$ is odd. This condition holds if and only if $1+i$ does not divide $\alpha$. We say that a Gaussian integer $\alpha$ is {\it primary} if it satisfies the condition
$$
\alpha \equiv 1 \bmod 2(1+i).
$$
A primary element is necessarily odd. For any odd Gaussian integer $\alpha$ the set of its associates $\{\pm \alpha, \pm i \alpha\}$ contains exactly one primary element. A Gaussian integer $z = x+iy$ with $x$ and $y$ in $\Z$ is said to be {\it primitive} if the integers $x$ and $y$ are coprime.

An element of the set $\mathcal P^{\rm odd} := \{ 3, 5, 7, 11, \dots \}$ is called an {\it odd natural prime}. We denote by $\mathcal P^{\rm G}$ the set of the odd primary irreducible Gaussian integers. A Gaussian integer $z$ belongs to $\mathcal P^{\rm G}$ if and only if it satisfies exactly one of the two conditions
\begin{itemize}
\item $-z $ belongs to $\mathcal P^{\rm odd}$ and $z$ is congruent to $3 \bmod 4$,
\item $z$ is primary, $z \overline z$ belongs to $\mathcal P^{\rm odd}$ and is congruent to $1 \bmod 4$.
\end{itemize}

Any odd Gaussian integer $z$ is the product of a unit and of elements of $\mathcal P^{\rm G}$. This decomposition is unique up to the order. When $z$ is primary this unit is equal to $1$. The number of elements of $\mathcal P^{\rm G}$ appearing in this decomposition is denoted by $\widetilde \omega (z)$. In particular, if $n$ is an odd positive integer we have the equality
$$
\widetilde \omega (n) = 2 \omega_1 (n) +\omega_3 (n).
$$

We now give an easy decomposition of a positive integer which will be useful in \S \ref{transformation}.
\begin{lemma} 
\label{decomposition}
Let $n\geq 1$ be an odd squarefree integer and let $\beta_0$, $\beta_1$, $\beta_2$ and $\beta_3$ be four Gaussian integers such that 
\begin{equation}
\label{n=bbbb}
n= \beta_0 \beta_1 \beta_2 \beta_3.
\end{equation}
Then there exist 
\begin{itemize}
\item units $\eta_0$, $\eta_1$, $\eta_2$ and $\eta_3$,
\item positive integers $b_0$, $b_1$, $b_2$ and $b_3$, 
\item primitive Gaussian integers $z_{k,\ell}$ with $0 \leq k \not = \ell \leq 3$,
\end{itemize}
such that for $0\leq k \leq 3$ and $\ell\not= k$, one has the following properties
\begin{enumerate}[label=(\roman*)]
\item 
$$
\beta_k/b_k \text{ is a primitive Gaussian integer,}
$$
\item 
$$
\beta_k =\eta_k b_k \prod_{\ell\not= k} z_{k,\ell},
$$
\item 
$$
z_{\ell, k}=\overline{z_{k,\ell}},
$$
\item
$$
z_{k, \ell} \text{ is primary, }
$$
\item 
$$
\prod_{0\leq k \leq 3} \eta_k =1.
$$
\end{enumerate}
Finally, given $(n, \beta_0, \beta_1, \beta_2, \beta_3)$ satisfying \eqref{n=bbbb}, there is a unique set $\{ \eta_k, b_k, z_{k, \ell}\}$ satisfying the conditions {\it i)},...,{\it v)}.
\end{lemma}

\begin{remark}
In this decomposition no $z_{k, \ell}$ is divisible by some element of $\mathcal P^{\rm odd}$. The elements of the set $\{b_k, z_{k,\ell}\}$ are coprime in pairs since $n$ is squarefree. To lighten some notations, we will write $z_{k \ell}$ instead of $z_{k, \ell}$. Note that condition {\it v)} follows from the other conditions.
\end{remark}
\subsection{Sums of multiplicative functions}
We introduce the notation
$$
\LL := \log 2x.
$$ 
When bounding several error terms trivially, we will frequently use the following 
\begin{lemma}
\label{easy1} Let $\kappa >0$ be fixed. Then uniformly for $x \geq 1$, the following bounds hold true 
$$
\sum_{ n\leq x} \mu^2 (n) \kappa^{\omega (n)} \ll x\,\LL^{\kappa-1},
$$
$$
\sum_{ n\leq x} \mu^2 (n) \kappa^{\omega (n)}n^{-1} \ll \LL^{\kappa},
$$
and for $\ell =1$ or $3$
$$
\sum_{\substack{n\leq x \\ p \mid n \Rightarrow p \equiv \ell \bmod 4}} \mu^2 (n) \kappa^{\omega (n)} \ll x\, \LL^{\kappa/2-1},
$$
$$
\sum_{\substack{n\leq x \\ p \mid n \Rightarrow p \equiv \ell \bmod 4}} \mu^2 (n) \kappa^{\omega (n)}n^{-1} \ll \LL^{\kappa/2}.
$$
\end{lemma}

\subsection{Characters to detect squares} 
In definition \eqref{deff1} we need to detect whether a Gaussian integer is a square or not modulo a given Gaussian prime $\pi$. This detection will be accomplished by a character, which generalizes the Legendre symbol to the ring of Gaussian integers. If $\alpha$ is a non zero Gaussian integer, the number of residue classes of $\Z[i]$ modulo $\alpha \Z[i]$ is ${\mathrm N}(\alpha)$ and $\phi(\alpha)$ is the number of these classes which are coprime with $\alpha$. 
\begin{definition}
Let $\pi$ be an odd irreducible element of $\Z[i]$ and let $\alpha$ be an element of $\Z [i]$. Then we put
$$
\Bigl[ \frac \alpha \pi \Bigr] := \begin{cases} 
0 & \text{ if } \pi \mid \alpha,\\
1 & \text{ if } \pi \nmid \alpha \text{ and } \alpha \text{ is a square } \bmod \pi,\\
-1 & \text{ if } \pi \nmid \alpha \text{ and } \alpha \text{ is not a square } \bmod \pi.\\
\end{cases}
$$
\end{definition}
This character is sometimes denoted by $\displaystyle{\Bigl(\frac \alpha \pi \Bigr)_{\Q (i), 2}}$. It has the important 
property to be the square of the quartic character $\chi_{\pi} (\alpha)$
which is for instance defined in \cite[p. 122]{Ir-Ro} and it plays a central role in \cite[\S 4]{Fou-Klu-Ann}. We extend $[\frac \cdot \cdot]$ by multiplicativity to odd composite moduli $\beta$ factorized as a product of irreducible elements $\beta =\pi_1 \cdots \pi_s$, by the formula
$$
\Bigr[ \frac \alpha \beta \Bigr] = \Bigl[ \frac \alpha{\pi_1}\Bigr] \cdots \Bigl[ \frac \alpha{\pi_s}\Bigr],
$$
which is the analogue of the Jacobi symbol. 

We recall several formulas satisfied by the character $[\frac \cdot \cdot]$. The letter $\eta$ denotes an element of $\mathbb U$, the letter $\alpha$ denotes a Gaussian integer, the letter $\beta$ denotes an odd Gaussian integer and $\pi$ is an odd irreducible element of $\Z[i]$. We have
$$
\Bigl[ \frac \alpha \pi\Big] \equiv \alpha^{\frac {\rm N (\pi) -1}{2}}\bmod \pi,
$$
\vskip .2cm
\begin{equation}
\label{denom-eta}
\Bigl[ \frac \alpha \eta \Bigr] =1, \Bigl[ \frac \alpha{ \eta \beta} \Bigr] = \Bigl[ \frac \alpha \beta \Bigr],
\end{equation}
\vskip .2cm
\begin{equation}
\label{multiplicativity}
 \Bigl[ \frac {\alpha_1 \alpha_2}\beta \Bigr] = \Bigl[\frac{ \alpha_1} \beta \Bigr]\cdot \Bigl[\frac{ \alpha_2} \beta \Bigr],\
 \Bigl[ \frac {\alpha}{\beta_1 \beta_2} \Bigr] = \Bigl[ \frac {\alpha}{\beta_1 } \Bigr]\cdot \Bigl[ \frac {\alpha}{ \beta_2} \Bigr],
\end{equation}
\vskip .2cm
$$
\Bigl[ \frac {\alpha +\beta}\beta \Bigr] = \Bigl[ \frac {\alpha}\beta \Bigr] ,
$$
\begin{equation}
\label{-1/beta}
\Bigl[ \frac 1\beta\Bigr] = \Bigl[ \frac {-1} \beta \Bigr] =1, \ \Bigl[ \frac i \beta\Bigr] = \Bigl[ \frac {-i} \beta \Bigr] =
\begin{cases}
 1 & \text{ if } {\rm N}( \beta) \equiv 1 \bmod 8,\\
 -1 &\text{ if } {\rm N}( \beta) \equiv 5 \bmod 8,
 \end{cases}
\end{equation}
$$
\sum_{\alpha \bmod \beta} \Bigl[ \frac \alpha \beta\Bigr]
=\begin{cases} 0 & \text{ if } \beta \not= \eta \beta_1^2, \\
\phi (\beta) & \text { if } \beta = \eta \beta_1^2,
\end{cases}
$$
\begin{equation}
\label{conjugation}
\Bigl[ \frac{\overline{\alpha}}{\overline{\beta}}\Bigr] = \Bigl[ \frac \alpha{ {\beta}}\Bigr].
\end{equation}
If the real parts of $\alpha$ and $\beta $ are odd and if $\alpha$ is odd (for instance when $\alpha$ and $\beta$ are primary), we have the {\it reciprocity formula} due to Gauss (see \cite[Prop. 5.1]{Lemm-Book} for instance)
\begin{equation}
\label{reciprocity}
\Bigl[ \frac \alpha \beta\Bigr] = \Bigl[ \frac {\beta} {\alpha}\Bigr].
\end{equation}
If $n$ belongs to $\Z$ and if $\pi$ is such that $\pi\overline {\pi}=p$ belongs to $\mathcal P^{\rm odd}$, we have the equality
\begin{equation}
\label{legendre}
\Bigl[ \frac n \pi\Bigr] = \Bigl( \frac n p\Bigr),
\end{equation}
where the Legendre symbol appears on the right--hand side. If $a$ and $b$ are positive integers, with $(2a,b)=1$, we have the equality 
\begin{equation}
\label{=1always}
\Bigl[ \frac ab\Bigr] =1.
\end{equation}
\section{Oscillations of characters}
\subsection{Siegel--Walfisz type Theorems.} 
\begin{lemma}
\label{convenientSiegel} 
For every $A > 0$ one has the equality
$$
\sum_{n\leq x}\frac{ \mu^2 (nr)} {4^{\omega (n)}} \Bigl( \frac n q\Bigr) = O_A \bigr(\sqrt q \,x \,2^{\omega (r)}\, \LL^{-A}\bigl),
$$
uniformly over integers $r\geq 1$, $x \geq 2$ and odd squarefree integers $q>1$.
\end{lemma}

\begin{remark} 
In \cite[page 477]{Fou-Klu-Inv} such a sum is treated (with the constant $4$ replaced by $2$) however the proof is different: after restricting to integers $n$ with a reasonable number of prime factors, we apply the classical Siegel--Walsfisz Theorem to the largest prime factor. Such a device also appears in \cite[formula (80)]{Fou-Klu-Ann} and also in \cite[page 3631]{Fou-Klu-IMRN}.
\end{remark}

\begin{proof}
Consider the arithmetic function
$$
a(n)= a_{q,r}(n) := \frac{ \mu^2 (nr)} {4^{\omega (n)}} \Bigl( \frac n q\Bigr),
$$ 
and the associated Dirichlet series
$$
F(s) := \sum_{ n\geq 1} \frac {a (n)}{n^s}= \prod_{p\nmid r} \left( 1 + \frac {\bigl( \frac pq\bigr)}{4p^s}\right),
$$
considered as a function of the complex variable $s = \sigma +it$. This Dirichlet series is absolutely convergent for $\sigma >1$. Its expression as an Euler product leads to the formula 
$$
F(s) = G_r (s) \left\{L(s, (\cdot/q))\right\}^{1/4},
$$
where the function $G_r (s)$ is holomorphic on the half plane $\Re s >9/10$ and satisfies in this region the inequality $G_r (s) = O (2^{\omega (r)})$, and where the determination of $L(s, (\frac \cdot q))^{1/4}$ is chosen such that it tends to $1$ as $s$ is real and tends to $+ \infty$. It is well known that there exists a positive $c>0$ such that $L(s, (\frac \cdot q))$ has no zero in the region
\begin{equation}
\label{defOmega}
\Omega := \Bigl\{ s : \sigma > 1-\frac c{ \log (q (\vert t\vert +4))}\Bigr\}
\end{equation}
with at most one exception (Siegel's zero denoted by $\beta_1$), which, if it exists, is simple and located on the real axis. Furthermore, it satisfies the inequality
$$
\beta_1 < 1-\frac {c(\varepsilon)}{q^\varepsilon},
$$
where $\varepsilon >0$ is arbitrary and $c (\varepsilon ) >0$. See \cite[Theorems 11.3 \& 11.14]{Mo-Va-book} for instance.

We start from the equality
$$
\sum_{n \leq x} a(n)= \int_{2-i\infty}^{2+i \infty} F(s) x^s \frac {ds}s.
$$
If there is no Siegel zero $\beta_1$, we shift the above contour of integration to the  path $\mathcal G$ defined by the equality
$$
 \sigma=1-\frac {c/2}{ \log (q (\vert t\vert +2))},
$$
where $c$ is the constant appearing in \eqref{defOmega}. If $\beta_1$ exists, we replace the part of $\mathcal G$ satisfying 
$\vert t \vert \leq c(\varepsilon)/(2q^\varepsilon)$ by two horizontal segments with ordinates $\pm  c(\varepsilon)/(2q^\varepsilon)$
and a semi--circle with center $\beta_1$ and radius $c(\varepsilon)/(2q^\varepsilon)$. In both cases, all the zeroes of $L$ are on the left
of $\mathcal G$ and the function $F(s)$ is holomorphic on some open subset containing the part of the complex plane situated on the right--hand side of $\mathcal G$. To bound $\vert F(s)\vert $ on $\mathcal G$, we  appeal to  the following bounds  \cite[(11.6)]{Mo-Va-book}  or  \cite[(11.10)]{Mo-Va-book} for $L(s, (\frac \cdot q))$ according to the existence of $\beta_1$ and  to the situation of $s$ on $\mathcal G$ and we complete the proof of Lemma \ref{convenientSiegel}. This procedure is similar to the proof of the Siegel--Walfisz Theorem on sums of values of Dirichlet characters on consecutive primes.
\end{proof} 
 
\subsection{Double oscillations bounds for Jacobi symbols} Consider the bilinear sum over the Jacobi symbol
$$
\Omega (\boldsymbol \xi, \boldsymbol \zeta, M, N):=
\sum_{1\leq m \leq M} \sum_{1 \leq n \leq N} \mu^2 (2m) \mu^2 (2n)\xi (m) \zeta (n) \Bigl( \frac mn\Bigr),
$$
where $\boldsymbol \xi$ and $\boldsymbol \zeta$ are given sequences of complex numbers. We recall \cite[Lemma 15\,(18)]{Fou-Klu-Inv} (see also \cite[Prop. 10]{Fou-Klu-PLMS}).

\begin{lemma}
\label{billegendre}
Let $\xi (m)$ and $\zeta (n)$ be complex sequences with modulus less than $1$. Then, for every $ \varepsilon >0$, uniformly for $M$ and $N\geq 1$ we have
\begin{equation}
\label{sumsumllter}
\Omega (\boldsymbol \xi, \boldsymbol \zeta, M, N) \ll_\varepsilon MN (M^{-1/2 +\varepsilon} + N^{-1/2+\varepsilon}).
\end{equation}
\end{lemma}
This quite general lemma shows cancellation as soon as $\min (M, N)$ tends to infinity. Actually, we will use Lemma \ref{billegendre} under an extended form, where the number of divisors of the integer $n$ is denoted by $d(n)$.
\begin{lemma}
\label{billegendreextended}
Let $\xi (m)$ and $\zeta (n)$ be complex sequences, such that $\vert \xi(m) \vert \leq d(m)$ and $\vert \zeta(n)\vert \leq 1$ for all $m$ and $n \geq 1$. Then, for every $ \varepsilon >0$, uniformly for $K$, $M$ and $N\geq 1$ we have the inequality
\begin{equation*}
\Omega (\boldsymbol \xi, \boldsymbol \zeta, M, N) \ll_\varepsilon KMN (M^{-1/2 +\varepsilon } + N^{-1/2+\varepsilon}) + K^{-1} MN (\log M)^3.
\end{equation*}
\end{lemma}

\begin{proof} 
Of course, we could go to the original proof of Lemma \ref{billegendre} and insert, for some integer $r$, the $\ell_r$--norm of the sequence $\xi (m)$. We prefer to give a proof starting from Lemma \ref{billegendre} itself. We denote by $\Omega_{<K}$ the subsum of $\Omega (\boldsymbol \xi, \boldsymbol \zeta, M, N)$ corresponding to pairs $(m,n)$ such that $\vert \xi (m) \vert \leq K$ and $\Omega_{\geq K}$ is the complementary sum. So we have the equality $\Omega (\boldsymbol \xi, \boldsymbol \zeta, M, N) = \Omega_{<K} + \Omega_{\geq K}$. A direct application of \eqref{sumsumllter} gives the bound
$$
\Omega_{<K} \ll KMN (M^{-1/2 +\varepsilon } + N^{-1/2+\varepsilon}).
$$
The other sum $\Omega_{\geq K}$ is handled trivially by
$$
\vert \Omega_{\geq K} \vert \leq N \sum_{m\leq M \atop d(m) \geq K } d(m)\leq N \sum_{m\leq M} \frac{ d(m)^{2}}{K}\ll 
K^{-1}MN(\log M)^3.
$$
Adding these bounds completes the proof of the lemma.
\end{proof}
 
\subsection{Double oscillations bounds for $[\frac \cdot \cdot]$--symbols} We now consider the situation where, in the bilinear form, the Jacobi symbol is replaced by the $[\frac \cdot \cdot ]$--symbol, which turns out to be very similar.

To be more precise, let us define the bilinear form
$$
\Xi (\boldsymbol \xi, \boldsymbol \zeta, A,B) := \sum_{{\rm N} (\alpha) \leq A} \sum_{{\rm N} (\beta) \leq B} \xi (\alpha) \zeta (\beta) \Bigl[ \frac \alpha \beta\Bigr],
$$
where $\xi (\alpha)$ and $\zeta (\beta)$ are complex numbers defined on the set of odd Gaussian integers $\alpha$ and $\beta$. By a weaker form of \cite[Proposition 9]{Fou-Klu-Ann} we have
\begin{lemma} 
\label{boundsforbili2}
Let $\xi (\alpha)$ and $\zeta (\beta)$ be complex sequences
with support included in the set of primary squarefree Gaussian integers. Furthermore, suppose that these sequences satisfy the inequalities
\begin{equation*}
\vert \xi (\alpha)\vert, \ \vert \zeta (\beta)\vert \leq 1.
\end{equation*}
Then, uniformly for $A$ and $B\geq 1$, we have
$$
\Xi (\boldsymbol \xi, \boldsymbol \beta, A, B) \ll AB (A^{-1/9} +B^{-1/9}).
$$
\end{lemma}
The trivial bound for $\Xi$ is $O(AB)$. Any bound of $\Xi$ of the shape $\Xi \ll AB (A^{-\delta} + B^{-\delta})$ for some positive $\delta$
would be sufficient for the proof of Theorem \ref{central}. The same remark applies to \eqref{sumsumllter}.
 
\section{Proof of Theorem \ref{central}. First steps}
\subsection{Transformation of $f(n)$} \label{transformation} 
Our purpose is to use the character $[\cdot/\cdot]$ to transform the function $f (n)$ when $n$ is a positive squarefree integer. Recall the definition of $f(n)$, see \eqref{deff1}
\begin{multline*}
f(n) := \frac 14 \ \cdot \sharp \Bigl\{\beta \in \mathbb Z [i] :\ \beta \equiv \pm 1 \bmod 4\Z [i], \ \beta \vert n\text{ such that } \\ \text{ for all } \pi \vert \beta \text{ the Gaussian integer } n/\beta \text { is a square modulo } \pi \\
\text{ and for all } \pi \vert (n/\beta) \text{ the Gaussian integer } \beta \text { is a square modulo } \pi
\Bigr\}.
\end{multline*}
We further recall that $f(n)$ is equal to $2^{\rk_4 \text{Cl}(K_n)}$ for generic $n$ by Theorem \ref{tAlgebra}. First of all, the value of $f(n)$ does not change if, in the definition \eqref{deff1}, we restrict ourselves to primes $\pi $ belonging to
$\mathcal P^{\rm G}$. We detect the condition {\it for every $\pi \vert \beta, \text{ we have } \Bigl[ \frac {n/\beta}\pi\Bigr]=1$} by
\begin{equation}
\label{detect1}
\frac 1{2^{\widetilde \omega (\beta)}} \prod_{\pi \mid \beta\atop \pi \in \mathcal P^{\rm G}} \Bigl(1 + \Bigl[ \frac {n/\beta}\pi\Bigr]\Bigr) = \frac 1{2^{\widetilde \omega (\beta)}} \sum_{\beta_1 \mid \beta\atop \beta_1 \text{ primary } }
\Bigl[ \frac {n/\beta}{\beta_1}\Bigr],
\end{equation}
the value of which is $1$ or $0$. 

Similarly, we detect the condition {\it for every $\pi \vert n/\beta, \text{ we have } \Bigl[ \frac {\beta}\pi\Bigr]=1 $} by 
\begin{equation}
\label{detect2}
\frac 1{2^{\widetilde \omega (n/\beta)}} \prod_{\pi \mid n/\beta\atop \pi \in \mathcal P^{\rm G}} \Bigl(1 + \Bigl[ \frac {\beta}\pi\Bigr]\Bigr) = \frac 1{2^{\widetilde \omega (n/\beta)}} \sum_{\beta_3 \mid n/\beta\atop \beta_3 \text{ primary }}
\Bigl[ \frac {\beta}{\beta_3} \Bigr].
\end{equation}
Writing $\beta =\beta_0 \beta_1$ and $n/\beta = \beta_2 \beta_3$, gathering \eqref{detect1} and \eqref{detect2} and expanding the sums and the characters we finally obtain the equality
\begin{equation}
\label{f(n)=}
f(n)
=\frac 14 \sum_{\beta_0} \frac 1{2^{\widetilde \omega (\beta_0)}} \sum_{\beta_1} \frac 1{2^{\widetilde \omega (\beta_1)}} \sum_{\beta_2} \frac 1{2^{\widetilde \omega (\beta_2)}} \sum_{\beta_3} \frac 1{2^{\widetilde \omega (\beta_3)}} \Bigl[ \frac {\beta_0\beta_1}{\beta_3} \Bigr] \cdot 
\Bigl[ \frac {\beta_2\beta_3}{\beta_1}\Bigr] 
\end{equation}
where the sum is over $\boldsymbol \beta = (\beta_0, \beta_1, \beta_2, \beta_3)\in \Z[i]^4$ such that
\begin{equation}
\label{split}
n= \beta_0 \beta_1 \beta_2 \beta_3, \, \beta_0 \beta_1 \equiv \pm 1 \bmod 4, \, \beta_1 \text{ and }\beta_3 \text{ primary. }
\end{equation}
These congruence conditions imply that $\beta_1$ and $\beta_3$ both have odd real parts. Hence, by the reciprocity relation \eqref{reciprocity}, the equality \eqref{f(n)=} simplifies into
\begin{equation*}
f(n)
=\frac 14 \sum_{\beta_0} \frac 1{2^{\widetilde \omega (\beta_0)} }\sum_{\beta_1} \frac 1{2^{\widetilde \omega (\beta_1)}} \sum_{\beta_2} \frac 1
{2^{\widetilde \omega (\beta_2)}} \sum_{\beta_3} \frac 1{2^{ {\widetilde \omega (\beta_3)}}} \Bigl[ \frac {\beta_0}{\beta_3} \Bigr] \cdot 
\Bigl[ \frac {\beta_2}{\beta_1}\Bigr], 
\end{equation*}
where the $\beta_i$ satisfy \eqref{split}. Let 
$$
S(x):= \sum_{n\leq x} \mu^2 (2n) \Bigl( \frac {f(n)} { 2^{\omega_3 (n) -1}}\Bigr)
$$
be the sum appearing in \eqref{importantsum}. Inserting the factorization of the variable $n$ given in \eqref{split} we obtain the equality
\begin{multline}
\label{S(x)=2}
S (x) =\frac 12 \sum_{\beta_0} \frac 1{2^{\widetilde \omega (\beta_0)+\omega_3 (\beta_0)}} \sum_{\beta_1} \frac 1{2^{\widetilde \omega (\beta_1)+\omega_3 (\beta_1)} }\sum_{\beta_2} \frac 1{2^{\widetilde \omega (\beta_2)+\omega_3 (\beta_2)}} \\ \sum_{\beta_3} \frac 1{2^{\widetilde \omega (\beta_3) +\omega_3 (\beta_3)} }\Bigl[ \frac {\beta_0}{\beta_3} \Bigr] \cdot 
\Bigl[ \frac {\beta_2}{\beta_1}\Bigr],
\end{multline}
where the Gaussian integers $\beta_i$ are odd and satisfy the congruence conditions
\begin{equation} \label{congruences}
\beta_0 \beta_1 \equiv \pm 1 \bmod 4, \, \beta_1\text{ and } \beta_3 \text{ primary},
\end{equation}
the constraint
$$
\beta_0 \beta_1 \beta_2 \beta_3 \in \N \text{ and } 1\leq \beta_0 \beta_1 \beta_2 \beta_3 \leq x,
$$
and the coprimality condition
$$
(\beta_k, \beta_\ell) =1 \text{ for } 0\leq k < \ell \leq 3.
$$
In \eqref{S(x)=2} the function $\omega_3$ has naturally been extended to Gaussian integers $z$ by defining $\omega_3 (z)$ to be the number of irreducible divisors of $z$ belonging to $\mathcal P^{\rm odd}$.
\subsection{The main term.} Let $S^{\rm MT}(x)$ be the contribution to the right--hand side of \eqref{S(x)=2} coming from the $\boldsymbol \beta = (\beta_0, \beta_1, \beta_2, \beta_3)$ such that every $\beta_i$ is a non zero integer, of any sign. When the $\beta_i$ are odd integers, condition \eqref{congruences} simply becomes
\begin{equation}
\label{congsimpli}
\beta_1\equiv \beta_3 \equiv 1 \bmod 4.
\end{equation}
When $m$ is a non zero integer we have $\widetilde \omega (m) +\omega_3(m) = 2 \omega (m)$. Then we deduce the equality 
$$
S^{\rm MT} (x) = \frac 12 \sum_{1\leq n \leq x} \frac{\mu^2 (2n)}{4^{\omega (n)} }\cdot \nu (n),
$$
where $\nu (n)$ is the number of ways that $n$ can be written as $n = \beta_0 \beta_1 \beta _2 \beta_3$ with integers $\beta_i$ of any sign satisfying \eqref{congsimpli}. When $n$ is odd and squarefree, a direct computation shows the equality
$$
\nu (n) =2 \cdot 4^{\omega (n)}.
$$
Therefore we conclude that
\begin{equation}
\label{SMT=}
S^{\rm MT}(x) = \sum_{1\leq n \leq x} \mu^2 (2n),
\end{equation}
which corresponds to the first term on the right--hand side of \eqref{importantsum}.

\section{Preparation of the error term. Part I} 
Let $S^{\rm Err}(x) $ be the contribution to $S(x)$ of the terms $\boldsymbol \beta$ such that at least one $\beta_k$ (and hence at least two) is not an integer. Our goal is to prove that
\begin{equation}
\label{SErr=}
S^{\rm Err}(x) = O(x(\log x)^{-1/8}),
\end{equation}
which combined with \eqref{SMT=} will give the equality \eqref{importantsum} and hence Theorem \ref{central}.

\subsection{Factorization of the variables}
We appeal to Lemma \ref{decomposition} to factorize each Gaussian integer $\beta_k$ in \eqref{S(x)=2}. The summation over the four variables $\beta_k$ is replaced by twenty variables $\eta_k$, $b_k$, $z_{k\ell}$. We take time to precisely write this expression, where we exchanged the indices $1$ and $3$ in comparison with \eqref{S(x)=2}. We have
\begin{multline}
\label{S(x)=3} 
S^{\rm Err}(x)\\ = \frac 12 \sum_{\boldsymbol \eta} \sum_{\boldsymbol b} \frac 1{4^{\omega (\Pi \boldsymbol b)}} \sum_{\boldsymbol z}
 \frac {\mu^2 \bigl(2( \Pi \boldsymbol b) \,(\Pi \boldsymbol z)\bigr)} {2^{\widetilde \omega (\Pi \boldsymbol z)}} \Bigl[ \frac{\eta_0b_0 z_{01} z_{02}z_{03}}{\eta_1b_1z_{10} z_{12}z_{13}}\Bigr] \cdot \Bigl[ \frac{\eta_2b_2z_{20}z_{21} z_{23}}{\eta_3b_3 z_{30} z_{31} z_{32}} \Bigr],
\end{multline}
where $\Pi \boldsymbol b := b_0 b_1b_2 b_3$, $\Pi \boldsymbol z= \prod_{k\not= \ell } z_{k \ell}=\prod_{0\leq k < \ell\leq 3} \vert z_{k\ell}\vert^2$ and
\begin{itemize}
\item we have
\begin{equation}
\label{1<BZ<x}
1\leq (\Pi \boldsymbol b)( \Pi \boldsymbol z) \leq x,
\end{equation}

\item $\boldsymbol \eta = (\eta_0, \eta_1, \eta_2, \eta_3)$ belongs to $ \mathbb U^4$ and satisfies the equality
\begin{equation}
\label{prodeta}
\eta_0 \eta_1 \eta_2\eta_3 =1,
\end{equation}

\item $\boldsymbol b = (b_0,b_1, b_1, b_3)$ is a four--tuple of odd positive integers,

\item $\boldsymbol z = (z_{k \ell})_{0\leq k\not= \ell \leq 3}$ are primitive primary Gaussian integers, such that 
\begin{equation}
\label{conjugate} 
z_{k \ell} = \overline{ z_{\ell k}} \text{ for } 0\leq k < \ell \leq 3, 
\end{equation}

\item we have
\begin{equation}
\label{cong1}
\eta_0 \eta_3 b_0 b_3z_{01}z_{02} z_{31} z_{32} \vert z_{03}\vert^2\equiv \pm 1 \bmod 4, \ \eta_3 b_3\text{ and } \eta_1 b_1 \text{ are primary,}
\end{equation}

\item for some $0\leq k\leq 3$, we have 
\begin{equation}
\label{not=1}
\eta_k b_k\prod_{\ell \not= k} z_{k\ell} \not\in \Z.
\end{equation}
\end{itemize}

\subsection{Comments and simplifications of the formula \eqref{S(x)=3}.} Note that the factor $\mu^2 (2 (\Pi \boldsymbol b)( \Pi \boldsymbol z) )$ in the definition of $S^{\rm Err} (x)$ ensures that all the $b_k$ and all the $z_{k\ell}$ are odd and coprime by pairs. The integer $\Pi \boldsymbol z$ is only divisible by odd natural primes congruent to $1\bmod 4$ and and this remark leads to the equality 
\begin{equation}
\label{omega=omega}
\widetilde \omega (\Pi \boldsymbol z) = 2 \omega_1 (\Pi \boldsymbol z) =2 \omega (\Pi \boldsymbol z).
\end{equation}
 
Now consider the second part of \eqref{cong1}. Since $b_1$ and $b_3$ are positive integers the units $\eta_1$ and $\eta_3$ can only be equal to $\pm 1$. Hence the conditions $\eta_1 b_1$ and $\eta_3 b_3$ primary are equivalent to
\begin{equation}
\label{b1=eta1}
b_1 \equiv \eta_1 \text{ and } b_3 \equiv \eta_3 \bmod 4.
\end{equation}
Consider now the first part of \eqref{cong1}. Since $\vert z_{03}\vert^2$ is a positive integer $\equiv 1 \bmod 4$, since
$b_0$ and $b_3 $ are $\equiv \pm 1 \bmod 4$, since $\eta_3 =\pm 1$ and since the $z_{k\ell}$ are primary, we deduce that $\eta_0 \equiv \pm 1 \bmod 2(1+i)$ so we have $\eta_0= \pm 1$. Returning to \eqref{prodeta}, we deduce that $\eta_2= \pm 1$. So we have that 
\begin{equation}
\label{condforeta}
\boldsymbol \eta \in \{\pm1\}^4, \text{ and } \eta_0\eta_1 \eta_2 \eta_3 =1.
\end{equation}
With the above remarks, we see that the first part of \eqref{cong1} is equivalent to
\begin{equation}
\label{cong17}
z_{01}z_{02}z_{31}z_{32} \equiv \pm 1\bmod 4.
\end{equation}
Since the value of every $\eta_k$ is $\pm 1$, we see that \eqref{not=1} is equivalent to
\begin{equation}
\label{not=1ter}
\text{ for some } 0\leq k < \ell \leq 3 \text{ we have } z_{k\ell}\not= 1.
\end{equation}
That \eqref{not=1} implies \eqref{not=1ter} is clear. For the other direction suppose, for instance, that $b_0 z_{01}z_{02}z_{03} =b'$ where $b'$ is some integer and suppose that the primitive primary element $z_{01}$ is not equal to $1$. Then $z_{01}$ is divisible by some irreducible $\pi$, with $\pi \overline \pi $ an element of $\mathcal P^{\rm odd}$ congruent to $1$ modulo $4$. Necessarily $\overline \pi$ divides the integer $b'$ and hence $\overline \pi$ divides $b_0$, $z_{02}$ or $z_{03}$. But $\overline \pi$ does not divide the integer $b_0$ (otherwise $b_0$ and $z_{01}$ would not be coprime). So $\overline \pi$ divides $z_{02}$ for instance. But, by conjugation, $\overline \pi $ divides $\overline{z_{01}} = z_{10}$. So $z_{10}$ and $z_{02}$ would not be coprime and this is a contradiction.
 
Finally, by the values of the symbol $[\frac \cdot \cdot]$ given in \eqref{denom-eta} and \eqref{-1/beta} we can suppress the $\eta_k= \pm1 $ in the numerators and denominators of both symbols $[\frac \cdot \cdot ]$ in \eqref{S(x)=3}.
 
We benefit from all these remarks to simplify the formula \eqref{S(x)=3}. So we introduce the set $\mathcal U \subset (\Z[i]/4 \Z [i])^4$ defined by
$$
\mathcal U :=\bigl \{ (u_{01}, u_{02}, u_{13}, u_{23} ): u_{01}\,u_{02}\, \overline{u_{13}} \,\overline {u_{23}} \equiv \pm 1\bmod 4\bigr\}.
$$
After a decomposition of \eqref{cong17} into congruences modulo $4$ and a trivial summation over $\boldsymbol \eta$ and the $\boldsymbol b$ satisfying \eqref{b1=eta1} and \eqref{condforeta}, we split $S^{\rm Err} (x)$ into
\begin{equation*}
S^{\rm Err} (x) =\sum_{\boldsymbol u \in \mathcal U} S(x, \boldsymbol u),
\end{equation*}
with 
\begin{equation}
\label{S(x)=5} 
S(x , \boldsymbol u) = \sum_{\boldsymbol b} \frac 1{4^{\omega (\Pi \boldsymbol b)}} \sum_{\boldsymbol z}
\frac {\mu^2 \bigl(2( \Pi \boldsymbol b) \,(\Pi \boldsymbol z)\bigr)} {2^{\widetilde \omega (\Pi \boldsymbol z)}} \Bigl[ \frac{b_0 z_{01} z_{02}z_{03}}{b_1z_{10} z_{12}z_{13}}\Bigr] \cdot \Bigl[ \frac{b_2z_{20}z_{21} z_{23}}{b_3 z_{30} z_{31} z_{32}} \Bigr],
\end{equation}
where $\boldsymbol b$ and $\boldsymbol z$ satisfy \eqref{conjugate} and \eqref{1<BZ<x}, the condition \eqref{not=1ter} and the congruence conditions
\begin{equation}
\label{newcongcondition}
z_{01}\equiv u_{01}, z_{02}\equiv u_{02}, z_{13} \equiv u_{13}, z_{23} \equiv u_{23} \bmod 4.
\end{equation}
 
The sum $S (x, \boldsymbol u)$ contains ten independent variables of summation:
\begin{equation}
\label{variables}
b_0, b_1, b_2, b_3 \in \N \text{ and } z_{01}, z_{02}, z_{03}, z_{12}, z_{13}, z_{23} \in \Z[i],
\end{equation}
since the other $z_{k\ell}$ are linked to the other by \eqref{conjugate}. These variables do not have the same role: each variable $b_k$ appears exactly in one of the two symbols $[\frac \cdot \cdot]$, and thanks to \eqref{reciprocity} they are similar. The variable $z_{k\ell}$ and its conjugate $z_{\ell k}= \overline{z_{k \ell}}$ appear exactly once. But $z_{01}$ and $z_{10}$ appear in the numerator and in the denominator of the same symbol. The same remark is true for $z_{23}$ and $z_{32}$. The other $z_{k\ell}$ and $z_{\ell k}$ appear in different symbols. In its combinatorial aspect, this situation appears to be different from the one encountered in \cite{Fou-Klu-Inv} for instance.
 
\subsection{Trivial bounds for some subsums of $S (x, \boldsymbol u)$.} 
We first give a trivial bound for the complete sum $S (x, \boldsymbol u)$. Consider the equation \eqref{S(x)=5}. Since every $p\equiv 1 \bmod 4$ can be written in $12$ ways as
\begin{equation*}
p = \prod_{0\leq k \not= \ell \leq 3} z_{k\ell},
\end{equation*}
where the primitive primary Gaussian integers $z_{k\ell}$ satisfy the conjugacy condition \eqref{conjugate}, we deduce the following trivial inequality for $S (x, \boldsymbol u)$, where we bounded each character by $1$ and where we dropped the conditions \eqref{newcongcondition}: 
\begin{equation}
\label{crude1}
\vert S (x, \boldsymbol u) \vert \leq \sum_{\boldsymbol b}\frac 1{4^{\omega (\Pi \boldsymbol b)}}
\sum_{m\atop p \mid m \Rightarrow p \equiv 1 \bmod 4} \mu^2 (2 (\Pi \boldsymbol b) m)\cdot \frac{ 12^{\omega (m)}}{4^{\omega (m)}}.
\end{equation}
Here we used \eqref{omega=omega} and the sum is over the positive integers $\boldsymbol b =( b_0,b_1,b_1,b_3)$ and $m$ such that $(\Pi \boldsymbol b) m \leq x$. A direct application of Lemma \ref{easy1} implies the bound
\begin{align*} 
\vert S (x, \boldsymbol u) \vert &\ll\sum_{b\leq x}\mu^2 (2b) \frac {4^{\omega (b)}}{4^{\omega (b)}}
(x/b) \LL^{1/2} \nonumber\\
& \ll x \LL^{3/2}.
\end{align*}

As a consequence of \eqref{summu2}, we see that this crude bound of the error term is larger than $S^{\rm MT} (x)$ by a small power of $\LL$. 

We want to generalize this bound to some important subsums we will meet in the sequel of the proof. Let $\mathcal R$ be a set of positive integers less than $x$. Let $S_{\mathcal R} (x, \boldsymbol u)$ be the subsum of $S (x, \boldsymbol u)$ corresponding to the further restriction on the variables
$$
(\Pi \boldsymbol b) (\Pi \boldsymbol z) \in \mathcal R.
$$
We have
\begin{lemma} 
\label{15/4}
Uniformly for $x \geq 1$ and for  $\mathcal  R $ a subset of integers less than $x$, we have 
$$
S_{\mathcal R} (x, \boldsymbol u) \ll
(x \vert \mathcal R \vert)^{1/2} \LL^{15/4}.
$$
\end{lemma}

\begin{proof} 
Let $g$ be the multiplicative function defined on the set of odd squarefree integers by the formula
$$
g(p) = 
\begin{cases}
4 &\text{ if } p\equiv 1 \bmod 4,
\\
1 &\text{ if } p\equiv 3 \bmod 4.
\end{cases}
$$
By a computation similar to \eqref{crude1} and by the Cauchy--Schwarz inequality, we have the inequality
\begin{equation}\label{CAUCHY}
\bigl\vert S_{\mathcal R} (x, \boldsymbol u)\bigr\vert
\leq \sum_{r \in \mathcal R} g(r)\leq \bigl\vert \mathcal R \bigr\vert^{1/2} \Bigl( \sum_{n\leq x} g^2(n)\Bigr)^{1/2}.
\end{equation}
Let $h_1$ and $h_3$ be the two multiplicative functions defined on the set of squarefree integers by the formulas:
$$
h_1 (p)= \begin{cases}
16 &\text {if } p\equiv 1 \bmod 4,\\
0 &\text{if } p \equiv 3 \bmod 4,
\end{cases}
\text{ and }
h_3(p) =
\begin{cases}
0 & \text{ if } p\equiv 1 \bmod 4,\\
1 &\text{ if } p \equiv 3 \bmod 4, 
\end{cases}
$$
We have the convolution equality $g^2 = h_1 \star h_3$. It remains to apply Lemma \ref{easy1} twice to obtain
$$
\sum_{n\leq x} g^2 (n)\ll x \LL^{15/2}.
$$ 
By \eqref{CAUCHY} we complete the proof of Lemma \ref{15/4}.
\end{proof}

\section{Preparation of the error term. Part II} 
\subsection{Dissection of the domain of summation}
\label{dissectionsum}
We continue to prepare the error term $S (x, \boldsymbol u)$ by controlling the sizes of the ten variables appearing in \eqref{variables} and by removing the multiplicative constraint \eqref{1<BZ<x}. When this will be achieved, we will be in good position to apply Lemmas \ref{convenientSiegel}, \ref{billegendre} and \ref{boundsforbili2}. Let $\Delta$ be the the dissection parameter
$$
\Delta:= (1 + \LL^{-10}),
$$
We denote by $B_k$ and $Z_{k\ell}$ ($0 \leq k\not= \ell \leq 3$) any number taken in 
the set of powers of $\Delta$
$$
\{ 1, \Delta, \, \Delta^2, \Delta^3, \, \dots\}
$$
and we impose $Z_{k\ell} =Z_{\ell k}$, for $k \not= \ell$ as a consequence of \eqref{conjugate}. We define 
$$
\boldsymbol B := (B_0,\dots, B_3), \boldsymbol Z := (Z_{k\ell}), \Pi \boldsymbol B := B_0 B_1 B_2 B_3, \Pi \boldsymbol Z := \vert Z_{01} Z_{02} Z_{03} Z_{12} Z_{13} Z_{23}\vert^2.
$$ 
The notation $b_k\simeq B_k$ (resp. $z_{k\ell} \simeq Z_{k\ell} $) means that the integer variable of summation $b_k$ (resp. the primitive primary Gaussian integer $z_{k\ell}$) satisfies the inequalities $B_k \leq b_k < \Delta B_k$ (resp. $Z_{k\ell} \leq \vert z_{k\ell} \vert < \Delta Z_{k\ell}$). More generally the notation $\boldsymbol b \simeq \boldsymbol B$ means that, for each $0\leq k \leq 3$, we have $b_k \simeq B_k$. Then the notation $\boldsymbol z \simeq \boldsymbol Z$ has an obvious meaning. For $(\boldsymbol B, \boldsymbol Z)$ as above, we consider the \emph{cuboid}
\begin{equation}
\label{cuboid}
\mathcal C (\boldsymbol B, \boldsymbol Z):=
\prod_{0\leq k \leq 3}\Bigl[B_k, \Delta B_k\Bigr] \times \prod_{0\leq k \not= \ell\leq 3} \Bigl[Z_{k\ell}, \Delta Z_{k\ell} \Bigr].
\end{equation}

We return to the equality \eqref{S(x)=5}. We cover the set of summation defined by \eqref{1<BZ<x} by 
\begin{equation}
\label{O(LL110)}
O (\LL^{110})
\end{equation}
disjoint cuboids of the form $\mathcal C (\boldsymbol B, \boldsymbol Z)$.
 
If $\mathcal C (\boldsymbol B, \boldsymbol Z)$ is such that
$$
( \Pi \boldsymbol B )( \Pi \boldsymbol Z) \Delta ^{16} \leq x,
$$
then every element $(\boldsymbol b, \boldsymbol z)$ of $\mathcal C (\boldsymbol B, \boldsymbol Z)$ satisfies \eqref{1<BZ<x}.
 
In counterpart, if 
$$
( \Pi \boldsymbol B )( \Pi \boldsymbol Z)\leq x \text{ and } ( \Pi \boldsymbol B )( \Pi \boldsymbol Z)\Delta^{16} >x
$$
the elements $(\boldsymbol b, \boldsymbol z)$ of $\mathcal C (\boldsymbol B, \boldsymbol Z)$ do not necessarily satisfy \eqref{1<BZ<x}. However the contribution of these elements to $S (x, \boldsymbol u)$ is negligible. It suffices to apply Lemma \ref{15/4} with 
$$
\mathcal R =\bigl[x(1-O (\LL^{-10}), x \bigr],
$$
to see that contribution is in $ \ll (x^2\LL^{-10}) ^{1/2} \LL^{15/4} \ll x \LL^{-1/8},$ which is acceptable in view of \eqref{SErr=} that we want to prove.
 
Similarly the contribution to $S^{\rm Err} (x)$ of the union of the $ \mathcal C (\boldsymbol B, \boldsymbol Z)$ such that
$$
( \Pi \boldsymbol B )( \Pi \boldsymbol Z) \leq x \LL^{-10},
$$ 
is also negligible. To prove that, we apply Lemma \ref{15/4}, with $\mathcal R =[1, x \LL^{-10}]$ to see that this contribution is $\ll x \LL^{-1/8}$.

\subsection{The case of the cuboids with too many small edges.} 
Our purpose is to restrict our study to the cuboids $ \mathcal C (\boldsymbol B, \boldsymbol Z)$ which have at least four {\it large} edges. So we introduce the following

\begin{definition}
Let $\mathcal C (\boldsymbol B, \boldsymbol Z)$ be the cuboid defined in \eqref{cuboid}. Let $y$ be one of the ten independent variables of the list \eqref{variables} and let $[Y, \Delta Y]$ be the edge associated to this variable $y$. We say that this edge is large if 
\begin{enumerate}[label=(\roman*)]
\item 
$Y\geq \exp(\LL^{1/100})$ when $y$ is one of the $b_k$ ($0\leq k \leq 3$),
\item 
$ Y \geq \LL^{5000}$ when $y$ is one of the $z_{k\ell}$ ($0\leq k < \ell\leq 3$).
\end{enumerate}
If $Y$ does not satisfy these inequalities, we say that this edge is small.

Similarly, we say that the associated variable $y$ is large or small according to the inequality satisfied by $Y$.
\end{definition}

Let $S_{\geq 7} (x, \boldsymbol u)$ be the total contribution to $S (x, \boldsymbol u)$ of all the $ \mathcal C (\boldsymbol B, \boldsymbol Z)$ which have at least seven small edges associated to seven of the ten independent variables of the list \eqref{variables}. We prove

\begin{lemma}
\label{>7negligible}
For $x\geq 1$, we have
$$
S_{\geq 7} (x, \boldsymbol u) \ll x \LL^{-1/8}.
$$
\end{lemma}

\begin{proof} 
The definition of {\it small} depends on the variable considered and since the variables $b_k$ and $z_{k\ell}$ do not have the same role, we are obliged to consider different cases according to the respective number of $b_k$ and $z_{k\ell}$ which are large. However we only present the case where at most two $b_k$ (say $b_2$ and $b_3$) and at most one $z_{k\ell}$ (say $z_{23}$) is large. The other cases are similar. Returning to \eqref{S(x)=5}, we see that the total contribution (denoted by $\Sigma (x)$) to $S (x, \boldsymbol u)$ of the $ \mathcal C (\boldsymbol B, \boldsymbol Z)$ corresponding to this particular case satisfies the inequality 
\begin{multline*}
\vert \Sigma (x)\vert \leq \underset{b_0, b_1 \leq \exp (\LL^{1/100})} {\sum \sum  } \frac 1 {4^{\omega (b_0b_1b_2b_3 )}} \sum_{\vert z_{01}\vert ,\vert z_{02}\vert, \vert z_{03}\vert, \atop \vert z_{12}\vert ,\vert z_{13}\vert \leq \LL^{5000}} \frac 1{4^{\omega ( \vert z_{01}\cdots z_{13}\vert^2)}} \\
\underset{ b_2b_3\leq x/(b_0b_1 \vert z_{01}\vert^2\cdots)}{\sum \sum} \frac 1 {4^{\omega (b_2b_3)}} \sum_{\vert z_{23}\vert^2\leq x/(b_0b_1\cdots \vert z_{01}\vert^2 \cdots )} \frac 1{4^{\omega (\vert z_{23}\vert^2 )}},
\end{multline*}
where all the prime factors of the integers $\vert z_{k\ell}\vert^2$ are congruent to $1 $ modulo $4$. By the change of variables $b:=b_0b_1$, $m := \vert z_{01}\vert^2 \vert z_{02}\vert^2 \vert z_{03}\vert^2 \vert z_{12}\vert^2 \vert z_{13}\vert^2 $, $b'=b_2b_3$ and $m':=\vert z_{23}\vert^2$, we obtain the bound
\begin{multline*}
\vert \Sigma (x) \vert \\ \leq \sum_{ b \leq \exp(2 \LL^{1/100})} \frac 1{2^{\omega (b)}} 
\sum_{m\leq \LL^{50000} \atop p\mid m \Rightarrow p \equiv 1 \bmod 4} \Bigl( \frac 52\Bigr)^{\omega (m)} \sum_{b' \leq x/(bm) } \frac 1 {2^{\omega (b')}} \sum_{m' \leq x/(bb'm)\atop p\mid m' \Rightarrow p\equiv 1 \bmod 4} \frac 1{2^{\omega (m')}},
\end{multline*}
which is finally
$$
\Sigma (x) \ll x \LL^{-1/4+ 1/200+ \varepsilon},
$$
by a repeated application of Lemma \ref{easy1} and where $\varepsilon >0$ is arbitrary. This finally gives
$$
\Sigma (x) \ll x \LL^{-1/8}
$$
as desired.
\end{proof}

\subsection{The crucial sums} 
Let $(\boldsymbol B, \boldsymbol Z)$ be as in \S \ref{dissectionsum} and let 
$S(\boldsymbol B, \boldsymbol Z, \boldsymbol u)$ be the subsum of $S(x, \boldsymbol u)$ (see \eqref{S(x)=5})
defined by 
\begin{equation}
\label{S(x)=6} 
S(\boldsymbol B, \boldsymbol Z, \boldsymbol u) = \sum_{\boldsymbol b} \frac 1{4^{\omega (\Pi \boldsymbol b)}} \sum_{\boldsymbol z}
\frac {\mu^2 \bigl(2( \Pi \boldsymbol b) \,(\Pi \boldsymbol z)\bigr)} {2^{\widetilde \omega (\Pi \boldsymbol z)}} \Bigl[ \frac{ b_0 z_{01} z_{02}z_{03}}{ b_1z_{10} z_{12}z_{13}}\Bigr] \cdot \Bigl[ \frac{ b_2z_{20}z_{21} z_{23}}{b_3 z_{30} z_{31} z_{32}} \Bigr],
\end{equation}
where $\boldsymbol b = (b_k)_{0\leq k \leq 3}$ and $\boldsymbol z = (z_{k\ell})_{0\leq k\not= \ell \leq 3}$ satisfy \eqref{conjugate},
\eqref{not=1ter}, \eqref{newcongcondition} and 
\begin{equation*}
\boldsymbol b\simeq \boldsymbol B \text{ and } \boldsymbol z\simeq \boldsymbol Z.
\end{equation*}
Recall that the $b_k$ are positive integers and that the $z_{k\ell}$ are primitive primary Gaussian integers.
 
By the discussion developed in \S \ref{dissectionsum} we can suppose that $(\boldsymbol B, \boldsymbol Z)$ satisfies the inequalities
\begin{equation}
\label{sandwich}
x\LL^{-10} <( \Pi \boldsymbol B)(\Pi \boldsymbol Z)\leq x \Delta^{-16}
\end{equation}
By Lemma \ref{>7negligible}, we can restrict our study to the cuboids $\mathcal C (\boldsymbol B, \boldsymbol Z)$ with
\begin{equation}
\label{atleast4}
\text{ at least four large variables among the ten of the list \eqref{variables}. }
\end{equation} 
 
Finally, since the number of subsums $S(\boldsymbol B, \boldsymbol Z, \boldsymbol u)$ is bounded by \eqref{O(LL110)}, to prove the inequality \eqref{SErr=} it is sufficient to prove that for every $(\boldsymbol B, \boldsymbol Z)$ satisfying \eqref{sandwich} and \eqref{atleast4} and for every $\boldsymbol u \in \mathcal U$ one has the inequality
\begin{equation}
\label{bingo}
S(\boldsymbol B, \boldsymbol Z, \boldsymbol u) \ll x \LL^{-1/8-110}.
\end{equation}
 
\section{Proof of Theorem \ref{central}} 
The purpose of this section is to prove \eqref{bingo} by exploiting the oscillation of the character $[\frac \cdot \cdot]$ in different ways.
 
\subsection{Gymnastics on the product of two characters.}
\label{Gymnastics} 
Recall that $z_{k\ell}= \overline {z_{\ell k}}$. The ten independent variables given in the list \eqref{variables} appear in the formula \eqref{S(x)=6}. Let 
\begin{equation}
\label{defF(BZ)} 
F (\boldsymbol b, \boldsymbol z) := \Bigl[ \frac{b_0 z_{01} z_{02}z_{03}}{ b_1z_{10} z_{12}z_{13}}\Bigr] \cdot \Bigl[ \frac{b_2z_{20}z_{21} z_{23}}{b_3 z_{30} z_{31} z_{32}} \Bigr].
\end{equation}
First of all we want to factorize $F$ in a suitable way to apply bounds coming from Lemmas \ref{convenientSiegel}, \ref{billegendre} and \ref{boundsforbili2}. We will exploit the multiplicativity of the characters \eqref{multiplicativity}, and from the fact that all the elements $b_k$ and $z_{k\ell}$ have an odd real part to apply \eqref{reciprocity}. Finally, we will use the conjugation formula \eqref{conjugation}. To shorten formulas, we introduce the following notation: let $x$ be one of ten variables listed in \eqref{variables}, we denote by $ f(\widehat x)$ any function of the ten variables of \eqref{variables} but independent of $x$.

\begin{lemma} 
\label{cancell} 
Let $(x,y)$ be a pair of distinct variables taken in \eqref{variables} such that $(x,y)$ or $(y, x)$ belongs to the set $\mathcal E$ of twenty six pairs of variables defined by 
$$
\begin{matrix}
\mathcal E := &\{\ (b_0, b_1), & (b_0, z_{01}),& (b_0, z_{12}), &(b_0, z_{13}), 
 &(b_1, z_{01}),& (b_1, z_{02}),\\ 
 & (b_1, z_{03}), 
 &(b_2, b_3), & (b_2, z_{03}),& (b_2, z_{13}),& (b_2, z_{23}), 
& (b_3, z_{02}), \\ & (b_3,z_{12}), & (b_3,z_{23}), 
&(z_{01}, z_{02}),& (z_{01}, z_{03}), & (z_{01}, z_{12}), &(z_{01}, z_{13}), \\
& (z_{02},z_{03}), & (z_{02}, z_{12}), & (z_{02}, z_{23}), 
 & (z_{03}, z_{13}),& (z_{03},z_{23}),
&(z_{12}, z_{13}),\\
 & ( z_{12}, z_{23}), 
&(z_{13},z_{23})\ \}.
\end{matrix}
$$
Then, at least one of two following facts happen
\begin{enumerate}[label=(\roman*)]
\item there exist functions $\xi$ and $\zeta$ with modulus less than $1$ such that, for all the values of the variables $( \boldsymbol b, \boldsymbol z) $ we have
$$
F(\boldsymbol b, \boldsymbol z) = \xi (\widehat x) \zeta (\widehat y)\, \Bigl[ \frac xy\Bigr].
$$
\item there exist functions $\xi$ and $\zeta$ with modulus less than $1$ such that, for all the values of the variables $( \boldsymbol b, \boldsymbol z) $ we have
$$
F(\boldsymbol b, \boldsymbol z) = \xi (\widehat x) \zeta (\widehat y)\, \Bigl[ \frac x{\overline y}\Bigr].
$$
\end{enumerate}
\end{lemma}

\begin{proof} We only give the proof when $(x,y) =(z_{01}, z_{12})$. With obvious meanings for $\alpha$, $\beta$, $\gamma$ and $\delta$, write
\begin{equation}
\label{decomp20}
\Bigl[ \frac{b_0 z_{01} z_{02}z_{03}}{ b_1z_{10} z_{12}z_{13}}\Bigr] \cdot \Bigl[ \frac{b_2z_{20}z_{21} z_{23}}{b_3 z_{30} z_{31} z_{32}} \Bigr] 
= \Bigl[ \frac{\alpha z_{01}}{\beta \overline {z_{01} } z_{12}}\Bigr]\cdot
\Bigl[\frac {\gamma \overline{z_{12}}} {\delta}\Bigr],
\end{equation}
and decompose the first character as
$$
\Bigl[ \frac{\alpha z_{01}}{\beta \overline {z_{01} } z_{12}}\Bigr] = \Bigl[ \frac \alpha {\beta \overline{z_{01}}}\Bigr]\cdot \Bigl[ \frac \alpha {z_{12}}\Bigr] \cdot \Bigl[ \frac {z_{01}}{\beta \overline{z_{01}}}\Bigr] \cdot 
\Bigl[ \frac{z_{01}}{z_{12}}\Bigr].
$$
Combining with \eqref{decomp20}, the definitions of $\xi(\widehat {z_{01}})$ and $\zeta(\widehat {z_{12}})$ are obvious.
\end{proof}

\begin{remark}
\label{remark4}
Actually in the application below, we will never use the pairs $(b_k,b_\ell)$ of the set $\mathcal E$ since there is no the oscillation of the symbol $[\frac {b_k}{b_\ell}]$ because its value is $1$ always (see \eqref{=1always}).

Finally a pair $(z_{k\ell}, z_{k'\ell'})$ (with $k< \ell$, $k'< \ell'$ and $(k,\ell)\not=(k',\ell')$) belongs to the set $\mathcal E$ of Lemma \ref{cancell} if and only if the intersection of the set of indices $\{k, \ell\} \cap \{ k', \ell'\}$ contains exactly one element. This property implies that, in any set of three distinct variables $z_{k_i\ell_i}$ (with $1 \leq i\leq 3$ and $0\leq k_i < \ell_i\leq 3$), there exists at least two indices $i$ and $j$ such that $(z_{k_i\ell_i}, z_{k_j \ell_j})$ belongs to
$\mathcal E$.
\end{remark}

\subsection{The final steps}
Our proof of \eqref{bingo} is based on the number of large edges (at least four) of the cuboid $\mathcal C (\boldsymbol B, \boldsymbol Z)$ and the distribution of this number between the $b_k$ and the $z_{k\ell}$. Recall that the $z_{k\ell}$ are primary, primitive, squarefree and coprime by pairs. Our discussion is divided in four cases which do not exclude each other.
 
\subsubsection{The variables $b_0$, $b_1$, $b_2$ and $b_3$ are large and no $z_{k\ell}$ is large.} 
By the condition \eqref{not=1ter}, there is a $z_{k\ell}\not=1$. By symmetry, we can suppose that we have $z_{01}\not= 1$. By the multiplicative properties of the symbol $\Bigl[\frac \cdot \cdot\Bigr]$ and by \eqref{legendre} we factorize $F (\boldsymbol b, \boldsymbol z)$ defined in \eqref{defF(BZ)} as 
$$
F(\boldsymbol b, \boldsymbol z)= f(\widehat {b_0}) \Bigl[ \frac {b_0}{z_{10}z_{12} z_{13}}\Bigr] = f(\widehat {b_0})\Big(\frac {b_0}{\vert z_{01} z_{12}z_{13}\vert^2} \Bigr),
$$
where $f ( \widehat {b_0})$ is a function independent of $b_0$ of modulus less than one. Since the variables $z_{01}(\not =1)$, $\overline{z_{01}}$, $z_{12}$, $\overline{z_{12}}$,  $z_{13}$ and $\overline{z_{13}}$ are small, primitive, primary and coprime in pairs, the denominator 
$ \vert z_{01} z_{12}z_{13}\vert^2$ is a non square odd integer, satisfying the inequalities
\begin{equation}
\label{zzz<LL}
1 < \vert z_{10} z_{12}z_{13}\vert^2 \leq \LL^{30000}.
\end{equation}

We deduce the following inequality 
\begin{multline*}
\vert S(\boldsymbol B, \boldsymbol Z,\boldsymbol u)\vert
\leq \sum_{b_1\simeq B_1} \frac 1{4^{\omega (b_1)}}\sum_{b_2 \simeq B_2} \frac 1{4^{\omega (b_2)}}
\sum_{b_3 \simeq B_3} \frac 1{4^{\omega (b_3)}} \sum_{\boldsymbol z\simeq \boldsymbol Z} \frac 1{4^{\omega (\Pi \boldsymbol z)}}
\\
\Bigl\vert \sum_{b_0 \simeq B_0}
\frac {\mu^2 (b_0)}{4^{\omega (b_0)}} \Big(\frac {b_0}{\vert z_{10} z_{12}z_{13}\vert^2} \Bigr)
\Bigr\vert,
\end{multline*}
where, furthermore, $b_0$ is coprime with $r:= 2b_1b_2b_3 \vert z_{02} z_{03}z_{23}\vert^2$. We apply Lemma \ref{convenientSiegel} to the inner sum on $b_0$, with a very large $A$. Then we sum trivially over the other variables: by \eqref{sandwich}, by \eqref{zzz<LL} and by the inequality $\log B_0 \geq \LL^{1/100}$, we obtain \eqref{bingo} in that case.

\subsubsection{Three variables $b_k$  are large and some $z_{k'\ell'}$ is large.}
\label{3largeb}  We can suppose that the variables $b_0$, $b_1$ and $b_2$ are large. 
It easy to check that for any choice $0\leq k' < \ell' \leq 3$, at least one of the pairs $(b_0, z_{k'\ell'})$, $(b_1, z_{k'\ell'})$ and $(b_2, z_{k'\ell'})$ appears in the set $\mathcal E$ given in Lemma \ref{cancell}. To facilitate the exposition suppose that we are in the case where this pair is $(b_0, z_{01})$. Thanks to this lemma, we have the inequality 
\begin{equation}\label{decomp2}
\vert S (\boldsymbol B, \boldsymbol Z)\vert \leq \underset{b_1, b_2, b_3}{\sum\sum\sum}\ 
\underset{z_{02}, z_{03}, z_{12}\atop z_{13}, z_{23}}{\sum\cdots \sum} \Bigl\vert \sum_{b_0\simeq B_0} \sum_{z_{01} \simeq Z_{01}}
\xi (\widehat {b_0}) \zeta (\widehat{ z_{01}})\Bigl[ \frac {b_0} {z_{01}}\Bigr]
\Bigr\vert 
\end{equation}
Inspired by the equality
$$
\Bigl[ \frac {b_0} {z_{01}}\Bigr] = \Bigl( \frac {b_0}{\vert z_{01}\vert^2}\Bigl),
$$
we put $m:= \vert z_{01}\vert^2$. The number of ways of representing $m$ in this form is $O(d (m))$. Hence the last double sum in \eqref{decomp2} is of the shape $\Omega (\boldsymbol \xi', \boldsymbol \zeta, \Delta B_0, \Delta^2 Z_{01}^2)$ with $\vert \xi' (m) \vert \leq d(m).$ We apply Lemma \ref{billegendreextended}, with the choice $K = \LL^{150}$. By hypothesis $B_0$ and $Z_{01}^2$ are large so are both greater than $\LL^{10000}$. This lemma gives a non trivial bound for the last double sum (in \eqref{decomp2}) by a factor $\LL^{-147}$. Summing trivially over the variables $z_{02}, z_{03}, z_{12}, z_{13}, z_{23}, b_1, b_2, b_3$ in \eqref{decomp2}, we complete the proof of \eqref{bingo} in this case too.

\subsubsection{Two $b_k$ are large and two $z_{k'\ell'}$ are large} 
By directly checking all the possibilities for the two variables $b_k$ and the two variables $z_{k'\ell'}$ we claim that there is a pair $(b_{k_0}, z_{k'_0\ell'_0})$ of these large variables in the set $\mathcal E$. As soon as this pair $(b_{k_0},z_{k'_0\ell'_0}) $ is found, the proof is similar to \S \ref{3largeb}.
\subsubsection{Three $z_{k\ell}$ are large}
\label{Onebklarge}
By Remark \ref{remark4} of \S \ref{Gymnastics}, there exists a pair of these large variables $(z_{k\ell}, z_{k'\ell'})$ in the set $\mathcal E$
of Lemma \ref{cancell}. For simplicity of notations, suppose that this pair is $(z_{01}, z_{02})$. This allows us to rearrange $S (\boldsymbol B, \boldsymbol Z, \boldsymbol u)$ as follows 
\begin{equation*}
\vert S (\boldsymbol B, \boldsymbol Z, \boldsymbol u)\vert \leq \underset{b_0, b_1, b_2, b_3}{\sum\cdots\sum}\ 
\underset{z_{03}, z_{12}, z_{13} ,z_{23}}{\sum\cdots \sum} \Bigl\vert \sum_{z_{01}\simeq Z_{01}} \sum_{z_{02} \simeq Z_{02}}
\xi (\widehat {z_{01}}) \zeta (\widehat{ z_{02}})\Bigl[ \frac {z_{01} }{z_{02}}\Bigr]
\Bigr\vert ,
\end{equation*}
for some coefficients $\boldsymbol \xi$ and $\boldsymbol \zeta$ less than one in modulus. The double inner sum over $z_{01}$ and $z_{02}$ is of the form $\Xi(\boldsymbol \xi, \boldsymbol \zeta,\Delta^2 Z_{01}^2,\Delta^2 Z_{02}^2)$ which is studied in Lemma \ref{boundsforbili2}. Since $Z_{01}$ and $Z_{02}$ are larger than $\LL^{5000}$, this lemma gives a non trivial bound for the last double sum by a factor $\LL^{-10000/9}$. It remains to sum trivially over the $b_k$ and the four remaining $z_{k\ell}$ to obtain the bound \eqref{bingo} in this last case.
 
 
The proof of \eqref{bingo} has been accomplished in all the configurations of cuboids satisfying \eqref{sandwich} and \eqref{atleast4}. The proof of Theorem \ref{central} is now complete.

\section{Proof of Theorem \ref{main}}
\label{sMain}
We split
\[
\#\{0 < n < x : n \textup{ odd and squarefree}, \ \rk_4 \textup{Cl}(K_n) \neq \omega_3(n) - 1\}
\]
in the set of generic $n$ and its complement. The cardinality of the latter set is $O\bigl(x\log^{-1/4} \bigr)$, so it remains to bound
\[
g(x) := \#\{0 < n < x : n \textup{ odd, squarefree and generic}, \ \rk_4 \textup{Cl}(K_n) \neq \omega_3(n) - 1\}.
\]
By Theorem \ref{central}, there exists an absolute $C_0$, such that, for all $x\geq 2$, one has the inequality
$$
\sum_{n\leq x} \mu^2 (2n) \Bigl( \frac {f(n)}{2^{\omega_3 (n)-1}}\Bigr) \leq \sum_{n\leq x} \mu^2 (2n) +C_0\, x \log^{-1/8}x.
$$
By positivity, we deduce 
$$
\sum_{n\leq x\atop n \text{ generic} } \mu^2 (2n) \Bigl( \frac {f(n)}{2^{\omega_3 (n)-1}}\Bigr) \leq \sum_{n\leq x} \mu^2 (2n) + C_0\, x \log^{-1/8}x.
$$
Therefore, upon taking the difference of the above two equations, we obtain
\begin{align}
\sum_{n\leq x\atop n \text{ generic} } \mu^2 (2n) \Bigl( \frac {f(n)}{2^{\omega_3 (n)-1}}-1\Bigr)  &\leq \sum_{n\leq x\atop n \text{ not generic}} \mu^2 (2n) + C_0\,x \log^{-1/8}x,\nonumber\\
&\leq 2C_0 \,x \log^{-1/8} x, \label{theend}
\end{align}
for sufficiently large $x$. We appeal to Theorem \ref{tAlgebra} to conclude that 
$$
\frac {f(n)}{2^{\omega_3 (n)-1}}-1 \geq 0
$$ 
and that it is equal to zero if and only if $\rk_4 \text{Cl}(K_n) = \omega_3 (n)-1$ and that is $\geq 1$ if $\rk_4 \text{Cl}(K_n) \geq  \omega_3 (n)$. These remarks imply that the left--hand side of the inequality \eqref{theend} is larger than $g(x)$. This completes the proof of Theorem \ref{main}.


\end{document}